\documentclass[11pt]{amsart}
\usepackage{amsmath, amsthm, amssymb, amsfonts}
\usepackage[normalem]{ulem}
\usepackage{hyperref}

\usepackage{xypic, mathtools}
\usepackage{verbatim} 
\usepackage{longtable} 

\usepackage{mathtools}

\usepackage{caption}

\usepackage{tikz-cd}
\usetikzlibrary{arrows}

\theoremstyle{plain}
\newtheorem{thm}{Theorem}
\newtheorem{cor}{Corollary}
\newtheorem{lemma}{Lemma}
\newtheorem{lem}[lemma]{Lemma}

\newtheorem{prop}{Proposition}

\theoremstyle{definition}
\newtheorem{defn}{Definition}

\newtheorem{question}{Question}

\theoremstyle{remark}

\usepackage{tikz-cd}
\usetikzlibrary{arrows}

\newcommand{\BC}{{\mathbb{C}}}
\newcommand{\BD}{{\mathbb{D}}}

\newcommand{\BH}{{\mathbb{H}}}

\newcommand{\BQ}{{\mathbb{Q}}}
\newcommand{\BR}{{\mathbb{R}}}

\newcommand{\BT}{{\mathbb{T}}}

\newcommand{\BZ}{{\mathbb{Z}}}

\newcommand{\CA}{{\mathcal A}}
\newcommand{\CB}{{\mathcal B}}
\newcommand{\CC}{{\mathcal C}}

\newcommand{\CE}{{\mathcal E}}
\newcommand{\CF}{{\mathcal F}}
\newcommand{\CG}{{\mathcal G}}

\newcommand{\CI}{{\mathcal I}}

\newcommand{\CK}{{\mathcal K}}
\newcommand{\CL}{{\mathcal L}}
\newcommand{\CM}{{\mathcal M}}

\newcommand{\CO}{{\mathcal O}}
\newcommand{\CP}{{\mathcal P}}

\newcommand{\CR}{{\mathcal R}}

\newcommand{\CV}{{\mathcal V}}

\newcommand{\Fm}{{\mathfrak{m}}}

\newcommand{\blangle}{\big\langle}
\newcommand{\brangle}{\big\rangle}

\newcommand{\ch}{{\mathrm{ch}}}

\DeclareFontFamily{OT1}{rsfs}{}
\DeclareFontShape{OT1}{rsfs}{n}{it}{<-> rsfs10}{}
\DeclareMathAlphabet{\curly}{OT1}{rsfs}{n}{it}
\renewcommand\hom{\curly H\!om}

\newcommand\Ext{\operatorname{Ext}}
\newcommand\Hom{\operatorname{Hom}}
\newcommand{\p}{\mathbb{P}}

\newcommand\Spec{\operatorname{Spec}}

\newcommand*\dd{\mathop{}\!\mathrm{d}}
\newcommand{\Mbar}{{\overline M}}

\newcommand{\Coh}{\mathrm{Coh}}
\newcommand{\Pic}{\mathop{\rm Pic}\nolimits}
\newcommand{\PT}{\mathsf{PT}}

\newcommand{\GW}{\mathsf{GW}}

\usepackage{tikz}
\usepackage{lmodern}
\usetikzlibrary{decorations.pathmorphing}
\tikzset{snake it/.style={decorate, decoration=snake}}

\begin{document}
\title[Curve counting on elliptic Calabi--Yau threefolds]{Curve counting on elliptic Calabi--Yau threefolds
via derived categories}
\date{\today}

\author{Georg Oberdieck}
\address{MIT, Department of Mathematics}
\email{georgo@mit.edu}

\author{Junliang Shen}
\address{ETH Z\"urich, Department of Mathematics}
\email{junliang.shen@math.ethz.ch}

\begin{abstract}
We prove the elliptic transformation law
of Jacobi forms for the generating series
of Pandharipande--Thomas invariants of an
elliptic Calabi--Yau 3-fold
over a reduced class in the base.
This proves part of a conjecture by
Huang, Katz, and Klemm.
For the proof we construct an involution
of the derived category and use wall-crossing methods.
We express the generating series
of PT invariants
in terms of low genus Gromov--Witten invariants
and universal Jacobi forms.

As applications we prove new formulas and
recover several known formulas for
the PT invariants of
$\mathrm{K3} \times E$, abelian 3-folds, and the STU-model.
We prove that the generating series of curve counting
invariants for $\mathrm{K3} \times E$
with respect to a primitive class on the $\mathrm{K3}$ is a quasi-Jacobi form of weight -10.
This provides
strong evidence for the Igusa cusp form conjecture.
\end{abstract}

\baselineskip=14.5pt
\maketitle

\setcounter{tocdepth}{1} 

\tableofcontents
\setcounter{section}{-1}

\section{Introduction}
\subsection{Overview}
Recently, considerations in topological string theory
led Huang, Katz, and Klemm \cite{HKK}
to conjecture
a deep connection between
curve counting invariants of elliptic Calabi--Yau 3-folds
and the theory of Jacobi forms \cite{EZ}.
Strong modular structure results for all genus are predicted.

On the other hand,
derived categories of coherent sheaves
play a crucial role in counting curves
by the work of Pandharipande and Thomas \cite{PT1}.
Symmetries in derived categories are expected
to affect curve counting invariants.
More precisely,
Toda asks in \cite{TP2} the following question.

\begin{question}
How are curve counting invariants on a Calabi--Yau 3-fold constrained due to the presence of non-trivial derived auto-equivalences?
\end{question}

The purpose of this paper is to study Question 1
for elliptic Calabi--Yau 3-folds and
explore its connection to the
modular constraints
of the Huang--Katz--Klemm conjecture.
We prove that a specific auto-equivalence induces
the elliptic transformation law of Jacobi forms
on generating series.
By further combining techniques from
Jacobi forms, Pandharipande--Thomas theory, and Gromov--Witten theory
this yields new practical computation methods for curve counting
invariants. In particular we obtain strong evidence
for the Igusa cusp form conjecture \cite{K3xE}.

\subsection{Elliptic fibrations}
Let $X$ be a non-singular projective threefold satisfying
\[
\omega_X \simeq \CO_X \quad \text{and} \quad
H^1(X, \CO_X) = 0 \,.
\]
Let $S$ be a non-singular projective surface
and assume $X$ admits an elliptic fibration over $S$ --- a flat and proper morphism
\[
\pi : X \rightarrow S
\]
with fibers reduced and irreducible curves of arithmetic genus $1$.
The fibers of $\pi$ are therefore one of the following types:
\begin{enumerate}
    \item[(i)] a non-singular elliptic curve,
    \item[(ii)] a nodal rational curve (a cubic in $\p^2$ with a nodal singularity),
    \item[(iii)] a cuspidal rational curve (a cubic in $\p^2$ with a cuspidal singularity).
\end{enumerate}
We further assume the fibration $\pi$ has a section
\[ \iota: S \to X, \quad \pi \circ \iota = \mathrm{id}_S \,.
\]

\subsection{Stable pairs}
A stable pair $(\mathcal{F},s)$ on $X$ is a coherent sheaf $\mathcal{F}$
supported in dimension $1$ and a section
$s \in H^0(X, \mathcal{F})$ satisfying the following stability
conditions:
\begin{enumerate}
\item[(i)] the sheaf $\mathcal{F}$ is pure
\item[(ii)] the cokernel of $s$ is $0$-dimensional.
\end{enumerate}
To a stable pair we associate the Euler characteristic and
the class of the support $C$ of $\mathcal{F}$,
\[
\chi(\mathcal{F})=n\in \mathbb{Z}
\  \ \ \text{and} \ \ \ [C]= \beta \in H_2(X,\mathbb{Z})\,.
\]
Let $P_n(X, \beta)$ be the moduli space
of stable pairs of given numerical type.
Pandharipande--Thomas invariants \cite{PT1}
are defined by integrating
the Behrend function \cite{B}
\[ \nu : P_n(X, \beta) \to \BZ \]
with respect to the
topological Euler characteristic $e( \cdot )$
over the moduli space:
\[ \mathsf{P}_{n, \beta}
= \int_{ P_n(X, \beta) } \nu \dd{e}
= \sum_{k \in \BZ} k \cdot e\left( \nu^{-1}(k) \right) \,.
\]
The set of invariants $\mathsf{P}_{n, \beta}$
are intricately related
to the number and type of algebraic curves in $X$
\cite{PT3}.

\subsection{Elliptic transformation law}
Let $H \in  \mathrm{Pic}(S)$
be an effective divisor class
of arithmetic genus
\[ h = 1 + \frac{1}{2} (H^2 + K_S\cdot H) \,. \]
Let $F \in H_2(X,\BZ)$
be the class of a fiber of $\pi$,
and consider the curve classes
\[ H + d F := \iota_{\ast}(H) + d F \, \in H_2(X,\BZ), \quad d \geq 0 \,, \]
where we have suppressed the cycle class map that takes
the divisor $H$ to its class in homology.
We define the generating series
of stable pairs invariants,
\[ \mathsf{PT}_H(q,t)
=
\sum_{d = 0}^{\infty} \sum_{n \in \BZ}
\mathsf{P}_{n, H + dF} \, q^n t^d \,.
\]

By a calculation of Toda \cite[Thm 6.9]{T12}
we have a complete evaluation in case $H = 0$,
\begin{equation} \label{Toda_result}
\mathsf{PT}_0(q,t)
=
\prod_{\ell, m \geq 1} (1- (-q)^{\ell} t^{m} )^{-\ell \cdot e(X)}
\cdot \prod_{m \geq 1} (1-t^m)^{-e(S)} \,.
\end{equation}

We consider here the case when $H$ is \emph{reduced},
that is, if in every decomposition $H = \sum_i H_i$
into effective classes, all of the $H_i$ are primitive. 

The following is the main result of the paper.

\begin{thm}\label{Theorem1}
Let $H \in \Pic(S)$ be reduced of arithmetic genus $h$.
Then the following equality of generating series holds:
\[
\frac{ \PT_H(q^{-1} t , t) }{ \PT_0(q^{-1}t,t) }
=
q^{2(h-1)} t^{-(h-1)}
\frac{ \PT_H(q , t) }{ \PT_0(q,t)} \,.
\]
\end{thm}
\vspace{10pt}

Let $\mathsf{Z}_H(q,t) = \mathsf{PT}_H(q , t) / \PT_0(q,t)$.
Then Theorem \ref{Theorem1} can be rewritten as
\[ \mathsf{Z}_H(q^{-1} t, t)
= q^{2(h-1)} t^{-(h-1)} \mathsf{Z}_H(q,t) \,.
\]
By \cite{Br1, T16} every series
$\sum_{n \in \BZ} \mathsf{P}_{n,H + dF} q^n$
is the Laurent expansion of a
rational function in $q$ invariant under the
variable change $q \mapsto q^{-1}$.
Considering $\mathsf{Z}_H(q,t)$ as an element in $\BQ(q)[[t]]$,
we therefore also have
\[ \mathsf{Z}_H(q^{-1},t) = \mathsf{Z}_H(q,t) \,. \]
Combining with Theorem \ref{Theorem1} we obtain
the following relationship of $\mathsf{Z}_H$
to the theory of Jacobi forms \cite{EZ}.
\begin{cor} \label{Corollary1} Let $H \in \Pic(S)$ be reduced
of arithmetic genus $h$.
Then $\mathsf{Z}_H(q,t) \in \BQ(q)[[t]]$
satisfies the elliptic transformation law
for Jacobi forms of index $h-1$, that is, for all $\lambda \in \BZ$
\[
\mathsf{Z}_H(q t^{\lambda},t)
= t^{-(h-1)\lambda^2} q^{-2 (h-1) \lambda} \mathsf{Z}_H(q,t) \,.
\]
\end{cor}

To emphasize, the equality of Corollary \ref{Corollary1}
holds only as an identity of elements in $\BQ(q)[[t]]$.
In contrast, Theorem~\ref{Theorem1} is an equality of
generating series and yields an identity
on the level of coefficients.

By physical considerations and
explicit calculations,
Huang, Katz and Klemm conjecture the series $\mathsf{Z}_H(q,t)$
to be a meromorphic Jacobi form
of index $h-1$ \cite{HKK}.
Jacobi forms must satisfy two different equations:
the elliptic and the modular transformation law.
Corollary~\ref{Corollary1} therefore
proves exactly half of the conjecture
of Huang-Katz-Klemm in case $H$ is reduced\footnote{
The non-reduced case of Theorem~\ref{Theorem1}
will be considered in \cite{OS2}.}.
We will come back to this below.

\subsection{Genus zero}
Let
$\CM_{\beta}$ be the moduli space of
one-dimensional stable sheaves $\CF$
with $\chi(\CF) = 1$ and $\ch_2(\CF) = \beta$.
Following \cite{K} the genus $0$ Gopakumar--Vafa
invariant in class~$\beta$ is the
Behrend function weighted Euler characteristic
\[ \mathsf{n}_{\beta} = \int_{\CM_{\beta}} \nu \dd{e} \,. \]
The invariant $\mathsf{n}_{\beta}$ is a virtual count
of rational curves in class $\beta$.

We define the genus $0$ potential in classes $H + dF$ by
\[ \mathsf{F}_H(t) = \sum_{d \geq 0} \mathsf{n}_{H+dF} t^d \,. \]
Let also
\[ \phi_{-2,1}(q,t)
= (q + 2 + q^{-1}) \prod_{m \geq 1}
\frac{(1+q t^m)^2 (1+q^{-1} t^m)^{2}}{(1-t^m)^4} \, \]
be the (up to scaling) unique
weak Jacobi form of weight $-2$ and index $1$,
see \cite[Thm 9.3]{EZ}\footnote{
The variables $z \in \BC, \tau \in \BH$ of \cite{EZ}
are related to $(q,t)$ by
$q=e^{2 \pi i (z+1/2)}$ and $t = e^{2 \pi i \tau}$.
Also our convention of
$\phi_{-2,1}$ (but not $\phi_{0,1}$ below)
differs from \cite{EZ} by a sign.}.

 \begin{thm} \label{Theorem_genus0} Let $H \in \Pic(S)$ be irreducible of arithmetic genus $h=0$. Then
\[
\frac{\PT_H(q,t)}{\PT_0(q,t)}\ = \ 
\mathsf{F}_H(t) \cdot \frac{1}{\phi_{-2,1}(q,t)}
\]
\end{thm}
\vspace{10pt}
If $H$ is irreducible of genus $0$ then every curve
in $X$ of class $H+dF$
consists of a section over a line
in the base together with vertical components.
Theorem~\ref{Theorem_genus0} then says
that the series $\PT_H$
is a genus~0 term (counting sections)
times a universal contribution coming from the fiber geometry.

\subsection{General case}
We have the following more general structure result.

Consider the Weierstra{\ss} elliptic function
\[ \wp(q,t) = - \frac{1}{12} + \frac{q}{(1+q)^2}
- \sum_{d \geq 1}\sum_{m|d} m \big((-q)^m - 2 + (-q)^{-m}\big) t^d \]
and let
\[
\phi_{0,1}(q,t)
= 12 \wp(q,t) \cdot \phi_{-2,1}(q,t)
\]
be the unique weak Jacobi form of weight $0$ and index $1$.

\begin{thm} \label{Theorem_highergenus}
Let $H \in \Pic(S)$ be a reduced class of arithmetic genus $h \geq 0$,
and let $n$ be the largest integer for which there
exist a decomposition $H = \sum_{i=1}^{n} H_i$
into effective classes.
Then there exist power series
\[ f_{-(n-1)}(t), \ldots, f_{h}(t) \in \BQ[[t]] \]
such that
\[
\frac{\PT_H(q,t)}{\PT_0(q,t)}
\ = \  
\sum_{i=-(n-1)}^{h} f_i(t) \cdot \phi_{-2,1}(q,t)^{i-1} \phi_{0,1}(q,t)^{h-i} \,.
\]
\end{thm}
\vspace{4pt}

By the conjecture of Huang, Katz and Klemm \cite{HKK}, we expect
$\PT_H/\PT_0$ to be a meromorphic Jacobi form
of index $h-1$ and some weight~$\ell$.
This is equivalent to saying that every
\[ f_i(t) \in \BQ[[t]] \]
in Theorem~\ref{Theorem_highergenus}
is a weak modular form of weight $\ell+2i-2$.

Similar to before
we may think of
the functions $f_i(t)$
as counting generalized genus~$i$ sections
over curves in the base\footnote{
The genus can be negative here if the curve is disconnected.}.
Concretely,
by \cite{HKK} and
the examples of Section~\ref{Section_Examples}
for every $i$ there should be
a natural splitting
\[ f_i(t) =
\prod_{m \geq 1} \frac{1}{(1-t^m)^{-12 ( K_S \cdot H )}} \cdot g_i(t) \,. \]
By calculations of Bryan and Leung \cite{BL}
the factor
\begin{equation}
\prod_{m \geq 1} \frac{1}{(1-t^m)^{-12 (K_S \cdot H)}}
\label{first_factor} \end{equation}
is the contribution to $f_i(t)$
of a fixed curve
with no vertical components.
The remaining factor
\[ g_i(t) \in \BQ[[t]] \]
is expected
to be a modular form
related to the jumping of the Picard rank in the fibers
of the family
\begin{equation*} \pi^{-1}(L), \ \ L \in |H| \,. \end{equation*}
We hope that an approach using Noether--Lefschetz theory \cite{GWNL}
can provide a pathway to the modularity of~$g_i(t)$.

\subsection{Gromov--Witten theory}
We now assume that $X$ satisfies the GW/PT correspondence
which relates Pandharipande--Thomas invariants to
Gromov--Witten invariants of $X$, see \cite[Conj 3.3]{PT1}.
By \cite{PaPix1, PaPix2} the correspondence holds when
$X$ is a complete intersection in a product of projective spaces.

The genus $g$ Gromov--Witten invariant of $X$ is defined by the integral
\[ \GW_{g, \beta}
= \int_{[ \Mbar_{g}(X, \beta) ]^{\text{vir}}} 
1
\]
where $\Mbar_{g}(X, \beta)$ is the
moduli space of genus $g$ stable maps to $X$ with connected domain, and $[ \, \cdot \, ]^{\text{vir}}$ is its virtual class.
Define the generating series of genus $g$ Gromov--Witten invariants
\[ \GW_H^{g}(t) = \sum_{d \geq 0} \GW_{g, H+dF} t^d \,. \]

\begin{prop} \label{Theorem_GW}
Assume the GW/PT correspondence holds for $X$,
and let $H \in \Pic(S)$ be irreducible of arithmetic genus $h$.

Then the series $f_i(t)$ of Theorem~\ref{Theorem_highergenus}
are effectively determined from the series $\GW_H^{i}(t), i=0, \ldots, h$
via the equality
\begin{equation} \label{eqGW}
\sum_{i=0}^{h} f_i(t) \phi_{-2,1}(q,t)^{i-1} \phi_{0,1}(q,t)^{h-i}
\equiv
\sum_{g=0}^{h} 
\mathrm{GW}_H^{g}(t) u^{2g-2} \ \ \mathrm{mod}\ u^{2h}
\end{equation}
under the variable change $q = -e^{iu}$.
\end{prop}
\vspace{5pt}

By inverting the system \eqref{eqGW} we find
$\PT_H$ is an universal linear combination
of the first $h+1$ Gromov--Witten series.
In particular the Gromov--Witten invariants
up to genus $h$
determine
the Gromov--Witten invariants of arbitrary genus.
As examples we consider the first few cases.

Under the assumptions of Theorem \ref{Theorem_highergenus}
the genus $0$ Gopakumar--Vafa invariants agree with the
genus $0$ Gromov--Witten invariants:
\[ \mathsf{n}_{H+dF} = \GW_{0,H+dF} \,. \]
Hence in genus $h=0$ we recover Theorem \ref{Theorem_genus0}.
In genus $h = 1$ Theorem \ref{Theorem_highergenus} yields
\begin{equation} \label{gen1_eqn}
\frac{\PT_H(q,t)}{\PT_0(q,t)}
= \GW_H^{0}(t) \cdot \wp(q,t) + \GW_H^{1}(t) \,.
\end{equation}
To state the genus $2$ case, let
$$E_{2k}(t) 
= 1 - \frac{4k}{B_{2k}} \sum_{d \geq 1} \sum_{m | d} m^{2k-1} t^d$$
be the Eisenstein series with $B_{2k}$ the Bernoulli numbers.
Then, for $h=2$,
\begin{equation} \label{gen2_eqn}
\begin{aligned}
\frac{\PT_H}{\PT_0}
= & \phantom{+}
\GW_H^0 \cdot \left(
\wp^2 \phi_{-2,1} + \frac{1}{12} E_2 \wp \phi_{-2,1} +
\Big(\frac{1}{288} E_2^2 - \frac{11}{1440} E_4\Big) \phi_{-2,1}
\right) \\
& + \GW_H^1 \cdot \left( \wp \phi_{-2,1} + \frac{1}{12} E_2 \phi_{-2,1} \right) \\
& + \GW_H^2 \cdot \phi_{-2,1} \,.
\end{aligned}
\end{equation}
where we omitted the dependence on $q,t$.

\subsection{An example: $\mathrm{K3} \times E$}
Let $S$ be a non-singular projective K3 surface,
and let $E$ be an elliptic curve.
Consider the product Calabi--Yau $X = S \times E$
elliptically fibered along the projection
to the first factor,
\[ \pi : X \longrightarrow S \,. \]
Let $0_E \in E$ be the zero
and fix the section
\[ \iota : S \to X,\ s \mapsto (s, 0_E) \,. \]

For a non-zero class $H \in \Pic(S)$ the
group $E$ acts on the moduli space $P_n(X, H+dF)$
by translation with finite stabilizers. 
\emph{Reduced} Pandharipande--Thomas invariants of $X$ are defined
by integrating the Behrend function $\nu$
over the quotient space
\[
\mathsf{P}^{\text{red}}_{n, H+dF} = \int_{ P_n(X, H+dF)/E } \nu \dd{e}
\]
where the Euler characteristic is taken in the orbifold sense.
We define the
generating series
of reduced invariants
\[
\PT^{\mathrm{red}}_{H}(q,t)
= \sum_{d \geq 0} \sum_{n \in \BZ}  
\mathsf{P}^{\text{red}}_{n, H+dF} q^n t^d \,.
\]
If $H$ is primitive,
the series $\PT^{\mathrm{red}}_{H}(q,t)$ depends
by deformation invariance
only on the arithmetic genus~$h$ of $H$ and we write
\[ \PT^{\mathrm{K3} \times E}_{h}(q,t)
= \PT^{\mathrm{red}}_{H}(q,t) \,. \]

The ring $\text{QMod}$ of holomorphic quasi-modular forms
is the free polynomial algebra
in the Eisenstein
series $E_2(t)$, $E_4(t)$ and $E_6(t)$,
\[ \text{QMod} = \BQ[ E_2, E_4, E_6 ] \,. \]
Recall the Jacobi forms $\phi_{-2,1}(q,t)$ and
$\phi_{0,1}(q,t)$. The ring
\[ \widetilde{\mathrm{Jac}} = \text{QMod}[ \phi_{-2,1}, \phi_{0,1} ] \]
carries a natural bigrading by index and weight,
\[
\widetilde{\mathrm{Jac}}
= \bigoplus_{m \geq 0} \bigoplus_{k \geq -2m} \widetilde{\mathrm{Jac}}_{k,m} \,,
\]
where $E_{2k}$ has weight $2k$ and index $0$,
and $\phi_{k,1}$ has weight $k$ and index $1$.

Define the modular descriminant\footnote{
We modify here the usual definition of $\Delta$
by a shift of $t$ to avoid making a similar
shift in the definition of $\PT_H$.
}
\[ \Delta(t) = \prod_{m \geq 1} (1-t^m)^{24} \,. \]

\begin{thm} \label{Theorem_K3xE} We have
\[
\PT^{\mathrm{K3} \times E}_{h}(q,t)
=
\frac{\Psi_h(q,t)}{\Delta(t) \phi_{-2,1}(q,t)}
\]
for a series $\Psi_h(q,t) \in \widetilde{\mathrm{Jac}}_{0,h}$
of index $h$ and weight $0$.
\end{thm}

Hence in the language of \cite{HilbK3}
we find
$\PT^{\mathrm{K3} \times E}_{h}(q,t)$ is
a quasi-Jacobi form of index $h-1$ and weight $-10$.
In particular
the full series $\PT_h^{\mathrm{K3} \times E}$
is determined from finitely many coefficients.

By basic Gromov--Witten calculations
we recover a result of Bryan \cite{Bryan-K3xE}.
\begin{thm} \label{Theorem_Bryan}
\begin{align*}
\PT^{\mathrm{K3} \times E}_{0}(q,t)
& = \frac{1}{\Delta(t) \phi_{-2,1}(q,t)} \\
\PT^{\mathrm{K3} \times E}_{1}(q,t)
& = \frac{24 \wp(q,t)}{\Delta(t)}
\end{align*}
\end{thm}

A complete evaluation of the invariants
$\mathsf{P}^{\text{red}}_{n,H+dF}$
was conjectured in \cite{K3xE},
motivated by physical predictions \cite{KKV}
and the calculations \cite{HilbK3}.
For primitive $H$ the conjecture takes the form
\begin{equation}
\sum_{h = 0}^{\infty}
\PT^{\mathrm{K3} \times E}_{h}(q,t) u^h
= \frac{1}{\chi_{10}(q, t, u)} \label{CONJ} \end{equation}
where $\chi_{10}$ is the Igusa cusp form --- a Siegel modular form of weight 10.
Theorem~\ref{Theorem_Bryan} verifies
this conjecture in cases $h=0$ and $h=1$,
and Theorem~\ref{Theorem_K3xE}
gives strong evidence for every genus $h$.

\subsection{Euler characteristics}
The proof of Theorem \ref{Theorem1}
is based on wall-crossing techniques
and applies also for the
(unweighted) Euler characteristic of the moduli spaces.
We state parallel results
for the unweighted case.

Define na{\"i}ve Pandharipande--Thomas invariants
as the Euler characteristic of the moduli space
of stable pairs,
\[ \widetilde{\mathsf{P}}_{n, \beta}
= e\big( P_n(X, \beta) ) \,. \]
We form the generating series
\[
\widetilde{\PT}_{H}(p,t)
=
\sum_{d = 0}^{\infty} \sum_{n \in \BZ}
\widetilde{\mathsf{P}}_{n, H + dF} \, p^n t^d
\]
where we use the variable $p$ instead of $q$.
By the same argument as in Toda's calculation \cite{T12} we have
\[
\widetilde{\PT}_{0}(p,t)
= \prod_{\ell, m \geq 1} (1- p^{\ell} t^{m} )^{-\ell \cdot e(X)}
\cdot \prod_{m \geq 1} (1-t^m)^{-e(S)} \,.
\]
\begin{thm} \label{Theorem_highergenus_EulerChar}
Let $H \in \Pic(S)$ be reduced of arithmetic genus $h$.
Then we have the equality
of generating series
\[
\frac{ \widetilde{\PT}_H(p^{-1} t , t) }{ \widetilde{\PT}_0(p^{-1}t,t) }
=
p^{2(h-1)} t^{-(h-1)}
\frac{ \widetilde{\PT}_H(p , t) }{ \widetilde{\PT}_0(p,t)} \,.
\]
Moreover, let $n$ be the largest integer for which there
exist an effective decomposition $H = \sum_{i=1}^{n} H_i$.
Then there exist $\widetilde{f}_i(t) \in \BQ[[t]]$ such that
\[
\frac{\widetilde{\PT}_H(p,t)}{\widetilde{\PT}_0(p,t)}
\ = \  
\sum_{i=-(n-1)}^{h} \widetilde{f}_i(t) \cdot \phi_{-2,1}(-p,t)^{i-1} \phi_{0,1}(-p,t)^{h-i} \,.
\]
\end{thm}

In case of (honest) PT invariants
we expected $\PT_H(q,t)$ to be a Jacobi form
and the functions $f_k(t)$ to be modular forms.
As we will see in Section~\ref{Section_Examples}
in the example of abelian threefolds,
the same does \emph{not} hold for $\widetilde{f}_k(t)$.

Theorem \ref{Theorem_highergenus_EulerChar} implies
that the na{\"i}ve and the honest Pandharipande--Thomas
invariants of $X$ are closely related under the
variable change $q=-p$. The only difference
arises from counting the section terms differently and
the Behrend function does not seem to play any larger role.
Our results are therefore in non-trivial agreement with
previous observations and conjectures
on the Behrend function in elliptic geometries
\cite{Bryan-K3xE, BrKo, BOPY}.

\subsection{Idea of the proof of Theorem \ref{Theorem1}} \label{Section_Intro_ProofThm1}
The identity of Theorem \ref{Theorem1}
arises from two separate steps:
applying a derived equivalence $\Phi_H$
to the moduli space of stable pairs,
followed by a wall-crossing in
the motivic Hall algebra.

The derived equivalence $\Phi_H$
sends a stable pair $\CO_X \to \CF$
to a \emph{$\pi$-stable pair},
which is a modification
of the usual definition of stable pairs
adapted to the elliptic geometry.
If the stable pair has numerical invariants
$(H+dF, n)$ then the transformed $\pi$-stable pair
has invariants
\[
(H + \widetilde{d}F, \widetilde{n}) = \big( H+ (h+d+n-1)F, -n-2h+2 \big) \,.
\]

The definition of $\pi$-stable pairs
is similar in philosophy to the modification
of stable pairs by Bryan-Steinberg \cite{BS} for
a crepant resolution $X \to Y$ that
contracts an exceptional curve.
While for BS-pairs we allow a
pair $\CO_X \to \CF$ to have
$1$-dimensional cokernel along the exceptional curve,
here we allow a $\pi$-stable pair to have $1$-dimensional cokernel 
supported on arbitrary fibers of the fibration.
A wall-crossing argument following Toda \cite{T16}
relates $\pi$-stable pairs invariants
to usual Pandharipande--Thomas invariants with a correction term
involving the count of semistable sheaves supported on fibers of $\pi$.
This yields naturally the term $\PT_0$
in the equation of Theorem \ref{Theorem1}.

The derived auto-equivalence $\Phi_H$ also
arises naturally from the elliptic fibration $\pi$.
Consider the fibered product
\begin{equation*} \label{fibproduct} X \times_S X \end{equation*}
over the base $S$,
and let $\CI_\Delta^{\ast}$ be the dual
of the ideal sheaf of the diagonal in $X \times_S X$.
Up to a normalization
$\CI_{\Delta}^{\ast}$ is the Poincar\'e
sheaf of the elliptic
fibration. Let also
\[ X \xleftarrow{q} X \times_S X \xrightarrow{p} X \]
be the natural projections to the first and second factor.
The Fourier--Mukai transform
$\phi_{\CI_{\Delta}^{\ast}}$
with kernel
$\CI_{\Delta}^{\ast}$ is
\[ \phi_{\CI_{\Delta}^{\ast}}(\CE) =
R q_{\ast}\big( p^{\ast}(\CE) \otimes \CI_{\Delta}^{\ast} \big),
\quad \CE \in D^b \Coh(X) \,. \]

For a line bundle $\CL \in \Pic(X)$ let
\[ \mathbb{T}_\CL (\mathcal{E}) =  \CL\otimes \mathcal{E} \]
be the twist by $\CL$ and let $\mathbb{D} : D^b\mathrm{Coh}(X) \rightarrow D^b\mathrm{Coh}(X)$ be the dual functor,
\[
\mathbb{D}(\mathcal{E}) = R\hom_X( \CE, \CO_X)\,.
\]
The auto-equivalence $\Phi_H$ is then defined as the composition
\[ \Phi_H
= \mathbb{D}
\circ \mathbb{T}_{\pi^\ast \CO_S(H)}
\circ \phi_{\CI_{\Delta}^{\ast}} \,. \]

The strategy of the proof is summarized in the following diagram,
where $P_n^{\pi}(X,\beta)$ denotes the moduli space of $\pi$-stable pairs.
\[
\begin{tikzcd}
P_n(X,H+dF)
\arrow{rrrd}{\text{Auto-equivalence } \Phi_H} 
\arrow[dd, dashed, swap, "\substack{\text{Elliptic}\\\text{transformation}\\
\text{law}}"] \\
&&& P_{\tilde{n}}^{\pi}(X, H + \tilde{d}F)
\ar[snake it]{dlll}{\text{Wall-crossing}}\\
P_{\tilde{n}}(X, H + \tilde{d}F)
\end{tikzcd}
\]

\subsection{Plan of the paper}
In Section~\ref{Section_Elliptic_CY3s} we recall
several basic facts on elliptic fibrations
and study sheaves supported on fibers of $\pi$.
In Section~\ref{Section_pi_stable_pairs}
we introduce $\pi$-stable pairs and prove
the wall-crossing.
In Section~\ref{Section_Derived_equivalence}
we study the derived equivalence. In particular,
if $H$ is reduced,
we prove a complex $I^{\bullet}$ is a stable pair
if and only if $\Phi_H(I^{\bullet})$ is a $\pi$-stable pair.
This completes the proof of Theorem~\ref{Theorem1}.
In Section~\ref{Section_Applications}
we prove Theorems \ref{Theorem_genus0}
and \ref{Theorem_highergenus},
and the corresponding Euler characteristic case.
In Section~\ref{Section_Examples}
we apply our methods to several examples including $\mathrm{K3} \times E$
and abelian threefolds.

\subsection{Conventions} We always work over $\BC$.
For a variety $X$, the canonical bundle (or sheaf)
is denoted $\omega_X$, and the canonical
divisor is $K_X = c_1(\omega_X)$.
The skyscraper sheaf at a point $x \in X$ is $\BC_x$.
The dual of a sheaf $\CF$ is $\CF^{\ast} = \hom_X(\CF, \CO_X)$,
and the derived dual of a complex $\CE \in D^b\mathrm{Coh}(X)$
is $\CE^{\vee} = R\hom_X(\CE, \CO_X)$.
If $i: X \hookrightarrow Y$ is a closed embedding
and $\CF$ is a sheaf on $X$,
then we write $\CF$ also for the pushforward $i_{\ast} \CF$ on $Y$.
For a sheaf $\CF$ and a divisor $D$ on $X$,
we let $\CF(D) = \CF \otimes \CO_X(D)$. 

\subsection{Acknowledgements}
The paper was started
when both authors were attending the workshop
{\em Curves on surfaces and threefolds}
at the Bernoulli center at EPFL Lausanne in June 2016.
Discussions with J.~Bryan
on elliptic geometries and the paper \cite{BS}
were extremely helpful.
We would also like to thank
F.~Greer, S.~Katz, M.~Kool, O.~Leigh,
D.~Maulik, T.~Padurariu, R.~Pandharipande, Y.~Toda, and Q.~Yin
for useful discussions.

J.~ S. was supported by grant ERC-2012-AdG-320368-MCSK in the group of R. Pandharipande at ETH Z\"{u}rich.

\section{Elliptic Calabi--Yau threefolds}
\label{Section_Elliptic_CY3s}
\subsection{Definition}
Let $X$ be a Calabi--Yau 3-fold --- a non-singular projective
threefold with trivial canonical bundle $\omega_X \simeq \CO_X$ and $H^1(X, \CO_X)=0$.
Let $S$ be a non-singular projective surface
and let
\[
\pi : X \rightarrow S
\]
be an elliptic fibration with reduced and irreducible fibers.
Hence fibers 
\[ X_s = \pi^{-1}(s),\  s \in S\]
are either non-singular elliptic curves or
rational curves with a single node or cusp.
Let further
\[ \iota: S \to X, \quad \pi \circ \iota = \mathrm{id}_S \]
be a section of $\pi$. Necessarily,
the section meets every fiber $X_s$ in a non-singular point.
We let
\[ 0_s = \iota(s) \in X_s \]
denote the distinguished point of the fiber $X_s$ over $s \in S$
and
\[ S_0 : = \iota_{\ast}{S} \]
the divisor on $X$ defined by the section.

\subsection{Compactified Jacobian}
Let $\hat{X}$ be the relative compactified Jacobian of $X$
parametrizing torsion free, rank 1 and degree 0 sheaves
on a fiber of $\pi: X \rightarrow S$.
Since $\hat{X}$ is a fine moduli space,
there exists a universal Poincar\'e sheaf on 
\[
X \times_S \hat{X}
\]
uniquely defined up to tensoring by
a line bundle pulled back from $\hat{X}$.
We let $\mathcal{P}$ be the unique Poincar\'e sheaf
satisfying the normalization
\[
\mathcal{P}|_{S_0 \times_S \hat{X}} \simeq \CO_{\hat{X}}.
\]
The sheaves $\CP$ and $\CP^{\ast}$ are flat
over both $X$ and $\hat{X}$, and
$\CP^{\vee} = \CP^{\ast}$ \cite[8.4]{BM}.

Since $\pi$ admits a section and has integral fibers,
we will identify $X$ with its compactified Jacobian $\hat{X}$
via the natural isomorphism
\begin{equation*}
X \xrightarrow{\ \cong\ } \hat{X},\ \ 
x \mapsto {\imath_s}_\ast \big(\Fm_x^{\ast} \otimes \CO_{X_s}(-0_s)\big)
\end{equation*}
where $\imath_s: X_s \hookrightarrow X$
is the inclusion of the fiber over $s = \pi(x)$ and
$\Fm_x$ is the ideal sheaf of $x$ in $X_s$.
Let 
\[ X \xleftarrow{p} X\times_S X \xrightarrow{q} X \]
denote the natural projections. 
The normalized Poincar\'e sheaf is then
\begin{equation*}
 \mathcal{P}
 = \mathcal{I}_\Delta^{\ast} \otimes p^\ast \CO_X(-S_0)
 \otimes q^\ast \CO_X(-S_0) \otimes q^\ast \pi^\ast \omega_S
 \end{equation*}
 where $\CI_\Delta$ is the ideal
 sheaf of the diagonal $\Delta : X \to X \times_{S} X$.

\subsection{Fourier--Mukai transforms}
Let 
\[
\phi_{\CK} : D^b\Coh(X) \to D^b\Coh(X),\ 
\CE \mapsto
Rq_\ast(p^\ast \CE\, \overset{L}{\otimes}\, \CK) \,.
\]
denote the Fourier--Mukai transform with kernel $\CK \in D^b\Coh(X \times_S X)$.

We are mostly interested in
the Fourier--Mukai transform $\phi_{\CP}$ with
kernel the Poincar\'e sheaf $\CP$.
We have the following facts.

\begin{lemma}
\label{Lemma_phiP_autoequivalence}
\begin{enumerate}
\item The transform $\phi_{\CP}$
is an auto-equivalence with inverse
$\phi_{\mathcal{Q}}$ where
$\mathcal{Q} := \mathcal{P}^\vee \otimes p^\ast \pi^\ast \omega_S^{\vee}[1]$.
\item Let $\mathrm{inv} : X \to X$
be the involution which,
under the identification $X \cong \hat{X}$,
sends a torsion free sheaf to its dual. Then
\[
\phi_{\CP} \circ \phi_{\CP}
= \mathrm{inv}^{\ast} \circ \BT_{\pi^{\ast} \omega_S}[-1] \,.
\]
\item Let $\BD$ be the dual functor. Then
\[ \mathbb{D}\circ \phi_\mathcal{P} = \phi_\mathcal{Q} \circ \mathbb{D} = \phi_\mathcal{P}^{-1} \circ \mathbb{D} \,. \]
\end{enumerate}
\end{lemma}
\begin{proof}
For (1) see \cite{BM}
or \cite[Prop 2.5]{BBRP2}.
Then use
$\phi_{\CP^{\ast}} = \mathrm{inv}^{\ast} \circ \phi_{\CP}$
to prove the second.
The third follows by relative
Grothendieck--Verdier duality applied to the morphism
$q: X\times _BX \rightarrow X$,
and using
$\omega_{X \times_{S} X} = p^{\ast} \pi^{\ast} \omega_S^{\vee}$.
\end{proof}

\begin{lemma} \hfill \label{Lemma_phiP}
\begin{enumerate}
\item $\phi_{\CP}( \CO_X ) = \iota_{\ast} \omega_S[-1]$
\item For every line bundle $\CM$ on $S$, we have
$\phi_{\CP}( \iota_{\ast} \CM )  = \pi^{\ast} \CM$.
\item $\phi_{\CP}( \CO_X(S_0) ) = \pi^{\ast} \omega_S(-S_0)$
\item If $n>0$ then $\phi_{\CP}(\CO_X(-nS_0))[1]$
is a locally free sheaf of rank $n$.
\end{enumerate}
\end{lemma}
\begin{proof}
(1) We have
$\CO_X = \phi_{\CP}(\CO_{S_0})$.
Hence, by Lemma \ref{Lemma_phiP_autoequivalence} (ii)
\[ \phi_{\CP}(\CO_X) = \phi_{\CP}(\phi_{\CP}(\CO_{S_0}))
= \pi^{\ast} \omega_S \otimes \CO_S[-1] = \iota_{\ast} \omega_S[-1] \,. \]
\noindent (2) Since tensoring with a line bundle pulled back
from $S$ commutes with $\phi_{\CP}$,
this follows from (1) by tensoring with $\pi^{\ast} \CM$.

\noindent (3)
The vanishing
$H^1(X_s, \CO_{X_s}(0_s)) = 0$
for all $s \in S$
implies $\phi_\CP(\CO_X(S_0))$ is concentrated in degree $0$.
Applying $\phi_{\CP}$ to the short exact sequence
\[ 0 \to \CO_X \to \CO_X(S_0) \to \CO_{S_0}(S_0) = \iota_{\ast} \omega_S \to 0 \]
and using (1) and (2) we obtain the exact sequence
\[ 0 \to \phi_{\CP}( \CO_X(S_0) ) \to \pi^{\ast} \omega_S
\to \iota_{\ast} \omega_S \to 0 \,. \]

\noindent (4)
This follows either directly using the base change formula,
or alternatively by induction over $n$
using the sequence
\[ 0 \to \CO_X(-n S_0) \to \CO_{X}(-(n-1)S_0) \to \iota_{\ast} \omega_S^{-(n-1)} \to 0 \,. \qedhere \]
\end{proof}

\subsection{Sheaves supported on fibers}
Let
\[ \alpha \in \Pic(S) \]
be a ample divisor on $S$ such that
$\alpha + K_S$ is ample. The induced divisor
\[ \lambda = \pi^{\ast} \alpha + S_0 \in \Pic(X) \]
is ample by an application of the Nakai--Moishezon criterion.

Consider the full subcategory of sheaves supported in dimension $\leq 1$,
\[ \Coh^{\leq 1}(X) \subset \Coh(X) \,. \]
The slope of a sheaf $A \in \Coh^{\leq 1}(X)$ is defined by
\[
\mu(A) = \frac{\chi(A)}{\ch_2(A) \cdot \lambda} \ \in (-\infty, \infty]
\]
with the convention that 0-dimensional sheaves have slope $+\infty$.
A sheaf~$B$ is stable (semistable) if
$\mu(A) \, {<} {(\leq)} \, \mu(B)$
for every proper subsheaf~$A \subset B$.

We say a sheaf $A \in \Coh^{\leq 1}(X)$ is \emph{supported on fibers of}
 $\pi$ if the reduced support of $A$ is contained
in the union of finitely many fibers $X_s, s\in S$.
Let
\[ \CC \subset \Coh^{\leq 1}(X) \]
be the full subcategory of sheaves supported
on fibers of $\pi$.
The existence and uniqueness of Harder--Narasimhan filtrations
in $\CC$ with respect to $\mu$ is
induced by the corresponding properties in $\mathrm{Coh}^{\leq 1}(X)$.

For an interval
$I \subset (-\infty, \infty]$ we let
$\CC_I$ denote the extension closure
of all semistable sheaves in $\CC$ with slope in $I$,
together with the zero object,
\begin{equation} \label{CCI} \CC_{I}
=
\left\langle A \in \CC\, \middle|\,
\begin{array}{c}
A \mbox{ is } \mu \mbox{-semistable} \\
\mbox{ with } 
\mu(A) \in I 
\end{array}
\right\rangle \cup\{0\}.
\end{equation}

The following proposition connects 
the notion of semi-stability
to the Fourier--Mukai transform $\phi_{\CP}$.
\begin{prop} \label{Prop_Action_on_Semistables_in_CC}
Let $A \in \CC$ be a semistable sheaf
of slope $\mu_0 = \mu(A)$.
With the convention $-1/\infty = 0$ and $-1/0 = \infty$
we have:
\begin{enumerate}
\item If $\mu_0 \in (0, \infty]$,
then $\phi_{\CP}(A)$ is a semistable sheaf in $\CC$ of slope $-1/\mu_0$.
\item If $\mu_0 \in (-\infty, 0]$, then
$\phi_{\CP}(A) = T[-1]$ for a semistable sheaf
$T \in \CC$ of slope~$-1/\mu_0$.
\end{enumerate}
\end{prop}
\begin{proof}
By taking a Jordan--H\"older filtration
we may assume $A$ is stable and hence
the pushforward of a stable sheaf from
a fiber $X_s$. The claim then follows
from \cite[Thm 2.21]{BK} or along the lines of
\cite[Thm 3.2]{BBRP2}.
\end{proof}

\section{$\pi$-stable pairs and wall-crossing} \label{Section_pi_stable_pairs}

\subsection{Stable and $\pi$-stable pairs}
A stable pair is the datum $(\CF,s)$ of a pure $1$-dimensional sheaf $\CF$
and a section $s \in H^0(X,\CF)$ with $0$-dimensional cokernel. Following \cite{PT1}
we identify a stable pair $(\CF,s)$ with the complex
\[ I^{\bullet} = [\CO_X \xrightarrow{s} \CF] \]
in $D^b \Coh( X )$ where $\CO_X$ sits in degree $0$.
A stable pair has class $\ch_2(\CF) = \beta$ and Euler characteristic
$n$ precisely if $\ch(I^{\bullet}) = (1,0, -\beta, -n)$.

Recall the following characterization of stable pairs
in \cite{Lo}, see also \cite[Defn. 3.1]{T16}.

\begin{lemma} \label{Lemma_Pairs_characterization}
An object $I^{\bullet} \in D^b \Coh(X)$
with $\ch(I^{\bullet}) = (1,0, -\beta, -n)$ is a stable pair if and only if
\begin{enumerate}
\item[(1)] $h^{i}(I^{\bullet}) = 0$ whenever $i \neq 0,1$.
\item[(2)] $h^0( I^{\bullet} )$ is torsion free
and $h^1(I^{\bullet})$ is $0$-dimensional.
\item[(3)] $\Hom( Q[-1], I^{\bullet} ) = 0$
for every $0$-dimensional sheaf $Q$.
\end{enumerate}
\end{lemma}

The definition of $\pi$-stable pairs
is parallel to the characterization above
but with $0$-dimensional sheaves
replaced by \emph{$\pi$-torsion sheaves}
which are defined as follows.

\begin{defn} \label{Defn_pi_torsion}
A sheaf $A \in \Coh^{\leq 1}(X)$ is $\pi$-torsion if the following conditions are satisfied:
\begin{enumerate}
\item $A$ is supported on fibers of $\pi : X \to S$,
\item $\phi_{\CP}(A(S_0))$ is a sheaf.
\end{enumerate}
\end{defn}

We define $\pi$-stable pairs with
respect to the elliptic fibration $\pi: X \rightarrow S$.
\begin{defn} 
An object $I^{\bullet} \in D^b \Coh(X)$
with $\ch(I^{\bullet}) = (1,0, -\beta, -n)$
is a $\pi$-stable pair if
\begin{enumerate}
\item[(1)] $h^{i}(I^{\bullet}) = 0$ whenever $i \neq 0,1$.
\item[(2)] $h^0( I^{\bullet} )$ is torsion free
and $h^1(I^{\bullet})$ is $\pi$-torsion.
\item[(3)] $\Hom( Q[-1], I^{\bullet} ) = 0$
for every $\pi$-torsion sheaf $Q$.
\end{enumerate}
\end{defn}

In \cite[4.2]{T16} Toda considered objects
in the derived category
which are characterized by the conditions of Lemma~\ref{Lemma_Pairs_characterization}
but with $0$-dimensional sheaves
replaced by certain sheaves supported in dimension $\leq 1$.
The associated invariants were termed $L$-invariants
and play a central role in the proof
of rationality of
Pandharipande--Thomas invariants.
Hence we may consider $\pi$-stable pairs
to be a variant of the objects
defining $L$-invariants.
In particular, the definition of
$\pi$-stable pair invariants
and the proof of the wall-crossing formulas discussed below
are parallel to the discussion in \cite{T16}
and we will be brief.

\subsection{Slope stability}
Let $\CL$ be a polarization on the 3-fold $X$.
The $\CL$-slope of a coherent sheaf $\CE \in \Coh(X)$ is
\begin{equation*}
\mu_{\CL}(\CE)
= \frac{c_1(\CE) \cdot \CL^2}{\mathrm{rank}(\CE)} \ 
\in \BQ \cup \{ \infty \}. \label{mu_omega_slope}
\end{equation*}
A sheaf $\CG \in \Coh(X)$ is $\mu_{\CL}$-stable (resp. $\mu_{\CL}$-semistable) if 
$\mu_{\CL}(\CE) \, {<} \, {(\leq)} \, \mu_{\CL}(\CG)$
for every proper subsheaf $\CE \subset \CG$ of strictly smaller rank.
For an interval $I \subset \BR \cup \{ \infty \}$ let
\[ \Coh_I(X) =
\left\langle A \in \Coh(X)\, \middle|\,
\begin{array}{c}
A \mbox{ is } \mu_{\CL} \mbox{-semistable} \\
\mbox{ with } 
\mu_{\CL}(A) \in I 
\end{array}
\right\rangle \cup\{0\}
\]
be the extension closure of semistable sheaves of slope in $I$,
together with the zero object.
Let
\[ \CA = \langle \Coh_{\leq 0}(X), \Coh_{>0}[-1] \rangle \,. \]
be the tilt of $\Coh(X)$
along the torsion pair
$(\Coh_{>0}(X), \Coh_{\leq 0}(X))$.
As in \cite[3.3]{T16}
we will work
inside the full abelian subcategory
\[
\CB
= \big\langle \Coh_0(X), \Coh^{\leq 1}(X)[-1] \big\rangle
\subset \CA \,.
\]

\subsection{$\pi$-stable pair invariants}
\label{Section_pi_stable_pair_invariants}
Let $\CM$ be the moduli stack of objects 
in $\CA$ with fixed Chern character $(1,0,-\beta,-n)$ and let
\begin{equation} \CP_n^{\pi}(X, \beta) \subset \CM \label{4343} \end{equation}
be the substack of $\pi$-stable pairs
with $\ch(I^{\bullet}) = (1, 0, -\beta, -n)$.

Completely parallel to the case of
$L$-invariants in \cite[4.2]{T16}
we have the following Lemma.

\begin{lemma} \label{Lemma_constructible}
The $\BC$-valued points of $\CP_n^{\pi}(X, \beta)$
form a constructible subset of $\CM$,
and
$\mathrm{Aut}(I^\bullet) = \BC^\ast$ for every $\pi$-stable pair $I^\bullet$.
\end{lemma}

We define the $\pi$-stable pair invariant
\[ \mathsf{P}^{\pi}_{n,\beta} \in \BZ \]
by taking the motive defined by the inclusion \eqref{4343},
multiplying by the motive $[\BC^{\ast}]$
and applying the integration map
of the motivic hall algebra of $\CA$, compare \cite[4.2]{T16}.
By Lemma~\ref{Lemma_constructible}
the invariant $\mathsf{P}^{\pi}_{n,\beta}$ is well-defined.

\subsection{Wall-crossing}
We follow closely the discussion in \cite[3.7]{T16}.
Let 
\[ \Coh_{\pi}(X) \subset D^b \Coh(X)
\quad \text{and} \quad
\Coh_{P}(X) \subset D^b \Coh(X)
\]
be the full categories of $\pi$-stable pairs
and stable pairs respectively,
and recall from \eqref{CCI}
the category $\CC_{I}$
of fiber sheaves
defined by the interval $I \subset \BR \cup \{ \infty \}$.

By Proposition \ref{Prop_Action_on_Semistables_in_CC}
a sheaf $A$ is $\pi$-torsion precisely if
it is an element of $\CC$ and
all its Harder-Narasimhan factors have slope $\mu > -1$.
Hence an argument identical to
the proof of \cite[Lem. 3.16]{T16}
shows the following Lemma.

\begin{lemma} \label{wallcrossing}
We have the following identity in $\CB$:
\[ 
\blangle \Coh_{P}(X), \CC_{(-1, \infty)}[-1] \brangle
\, = \,
\blangle \CC_{(-1, \infty)}[-1],
\Coh_{\pi}(X)
\brangle \,.
\]
\end{lemma}

Define the generating series of $\pi$-stable pair invariants
\begin{equation*} 
\mathrm{PT}^{\pi}_H(q,t)
= \sum_{d \geq 0} \sum_{n \in \BZ}
\mathsf{P}^{\pi}_{n, H + dF} q^n t^d
\end{equation*}
and let
\begin{equation*}
f(q,t)
= \prod_{\ell, m \geq 1} (1- (-q)^{\ell} t^{m} )^{-\ell \cdot e(X)} \,.
\end{equation*}

The main result of this section is the following
PT/$\pi$-PT correspondence.

\begin{prop} \label{PropWallCross} 
For every effective $H \in \Pic(S)$, we have
\[
\mathrm{PT}_H(q,t)
=
f( tq^{-1}, t)^{-1} f(q, t) \cdot \mathrm{PT}^{\pi}_H(q,t) \,.
\]
\end{prop}

\begin{proof}
Consider the
generalized Donaldson--Thomas invariant
\[
N_{n,\beta} \in \BQ
\]
counting semistable sheaves
of Chern character $(0,0, \beta, n)$, see \cite[3.6]{T16}.

By applying the integration map
of the motivic Hall algebra of $\CA$
to the identity in
Lemma \ref{wallcrossing} (see \cite[Thm 3.17]{T16} for details)
we obtain
\begin{equation} \label{WALLCROSS}
\PT_H(q,t)
=
\exp\bigg(
\sum_{ \substack{ d \geq 0, n \in \BZ \\ n/d \in (-1, \infty) }}
(-1)^{n-1} n \mathsf{N}_{n,dF} q^n t^d
\bigg) \cdot
\mathrm{PT}^{\pi}_H(q,t) \,.
\end{equation}
The first factor on the right-hand side is the contribution of 
$\CC_{(-1, \infty)}[-1]$.
By \cite[Thm 6.9]{T12} we have for all $d > 0$
\[ \mathsf{N}_{n,dF} = \sum_{k | (n,d) } \frac{-e(X)}{k^2} \,. \]
We find
\begin{multline*}
\exp\bigg(
\sum_{ \substack{ d > 0, n \in \BZ \\ n/d \in (0, \infty) }}
(-1)^{n-1} n \mathsf{N}_{n,dF} q^n t^d
\bigg)
=
\prod_{\ell, m \geq 1} (1- (-q)^{\ell} t^{m} )^{-\ell e(X)}
= f(q,t)
\end{multline*}
and
\[
\exp\bigg(
\sum_{ \substack{ d > 0, n \in \BZ \\ n/d \in (-1,0) }}
(-1)^{n-1} n \mathsf{N}_{n,dF} q^n t^d
\bigg) = f(q^{-1}t, t)^{-1} \,.
\]
Plugging into \eqref{WALLCROSS} the proof is complete.
\end{proof}

\section{The derived equivalence $\Phi_H$}
\label{Section_Derived_equivalence}
\subsection{Overview}
In this Section we use the derived equivalence $\Phi_H$
to prove the following Theorem.

\begin{thm}\label{FM0} Let $H \in \Pic(S)$
be effective and reduced. Then
\[
\mathsf{P}_{n, H+dF}
= \mathsf{P}^{\pi}_{-n-2h+2,\, H + (d+n+h-1)F}
\]
\end{thm}
\vspace{5pt}

We immediately deduce Theorem \ref{Theorem1} as a corollary.
\begin{proof}[Proof of Theorem \ref{Theorem1}]
Rewriting Theorem~\ref{FM0} in generating series we find
\[ \mathrm{PT}_H(q,t)
= t^{h-1} q^{-2(h-1)} \mathrm{PT}^{\pi}_H(q^{-1}t,t) \,. \]
By applying Proposition \ref{PropWallCross}
to the last term yields
\[
\mathrm{PT}_H(q,t)
= t^{h-1} q^{-2(h-1)}\frac{f(q,t)}{f(q^{-1}t,t)}\cdot \mathrm{PT}_H(q^{-1}t,t) \,.
\]
By \eqref{Toda_result} we have
\[
\frac{f(q,t)}{f(q^{-1}t,t)} = \frac{\mathrm{PT}_0(q,t)}{\mathrm{PT}_0(q^{-1}t,t)}
\]
which completes the proof.
\end{proof}

\subsection{Basic properties of $\Phi_H$}
Let $H \in \mathrm{Pic}(S)$ be an effective class.
\begin{defn} \label{defnn4} Define the autoequivalence
\begin{equation*}
\Phi_H = \mathbb{D}\circ \mathbb{T}_{\pi^\ast \CO_S(H-K_S)} \circ \mathbb{T}_{\CO_X(S_0)}\circ \phi_\mathcal{P} \circ \mathbb{T}_{\CO_X(S_0)}.
\end{equation*} 
For an element $\CE \in D^b\mathrm{Coh}(X)$ we write 
\[ \widetilde{\CE} = \Phi_H(\CE) \,. \]
\end{defn}

Definition~\ref{defnn4} agrees with
the definition of $\Phi_H$ in Section~\ref{Section_Intro_ProofThm1},
but is more convenient for computations since we work
with the normalized Poincar\'e sheaf $\CP$.
We have several basic Lemmas.

\begin{lemma}
For all $\CE, \CF \in D^b \Coh(X)$
\begin{enumerate}
\item[(1)] $\widetilde{\CE[i]} = \widetilde{\CE}[-i]$.
\item[(2)] $\mathrm{Hom}(\CE, \CF) \cong \mathrm{Hom}(\widetilde{\CF}, \widetilde{\CE})$.
\end{enumerate}
\end{lemma}

We show $\Phi_H$ acts on $D^b\mathrm{Coh}(X)$ as an involution.
\begin{lem} \label{inv}
$\Phi_H \circ \Phi_H = \mathrm{id}_{D^b\mathrm{Coh}(X)}$. \end{lem}
\begin{proof}
By Lemma~\ref{Lemma_phiP_autoequivalence} we have
$\mathbb{D}\circ \phi_\mathcal{P} = \phi_\mathcal{P}^{-1} \circ \mathbb{D}$.
Hence, with 
\[ \mathcal{L}: = O_X(S_0+\pi^\ast H - \pi^\ast K_S) \]
and since tensoring with a line bundle
pulled back from $S$ commutes with $\phi_{\CP}$ we have
\[
\Phi_H ^2 =
(\mathbb{D} \circ \mathbb{T}_{O_X(S_0)} \circ \phi_{\mathcal{P}} \circ \mathbb{T}_{\mathcal{L}}) \circ (\mathbb{D}\circ \mathbb{T}_{\mathcal{L}} \circ \phi_{\mathcal{P}} \circ \mathbb{T}_{O_X(S_0)})
= \mathrm{id}_{D^b\mathrm{Coh}(X)}. \qedhere
\]
\end{proof}

\begin{lemma} \hfill \label{Lemma_PHI_H_examples}
\begin{enumerate}
\item $\widetilde{\CO_X} = \CO_X(-\pi^{\ast} H)$
\item Let $x \in X$ be a point, let
$\imath_s : X_s \hookrightarrow X$ be the inclusion, and let
$\Fm_x$ be the ideal sheaf of $x$ in $X_s$. Then
\[ \widetilde{\BC_x} = \imath_{s \ast} \Fm_x[-2] \,. \]
\item For all $n > 1$ we have
\[ \Phi_H( \CO_X(-n S_0) ) = \CV_n[1] \]
for a locally free sheaf $\CV_n$ of rank $n-1$.
\end{enumerate}
\end{lemma}
\begin{proof}
(1) and (3) follow directly from Lemma \ref{Lemma_phiP} and the definition.
For (2) we have
\[
\mathbb{T}_{\CO_X( \pi^{\ast} (H- K_S) + S_0)}\circ \phi_\mathcal{P} \circ \mathbb{T}_{\CO_X(S_0)}( \BC_x )
= \imath_{s \ast} \Fm_x^{\ast} \,.
\]
Hence using relative duality and since $\Fm_x$ is reflexive on $X_s$ we get
\[ \widetilde{\BC_x} = \BD( \imath_{s \ast} \Fm_x^{\ast} )
= \imath_{s \ast}( \Fm_{x}^{\ast \vee} )[-2]
= \imath_{s \ast}( \Fm_x )[-2] \,. \qedhere
\]
\end{proof}

We prove $\Phi_H$ acts as expected on the level of cohomology.

\begin{lemma} \label{Lemma_Chern_character}
Assume $\CE \in D^b \Coh(X)$ has Chern character
\[ \ch( \mathcal{E} ) = (1,0, -(H+dF) , -n), \]
in $H^{2\ast}(X,\BQ) = \sum_i H^{2i}(X,\BQ)$, then
\[
\mathrm{ch}(\Phi_H(\mathcal{E})) = (1, 0 , -(H+(n+d+h-1)F), n+2h-2).
\]
\end{lemma}
\begin{proof} By a direct calculation as in \cite[5.2]{AHR},
or by calculating
$\widetilde{\CO_X}, \widetilde{\CO_{X_s}}, \widetilde{\BC_x}$, and $\widetilde{\CO_{\iota({B})}}$ with $B \in |H|$ directly as in Lemma \ref{Lemma_PHI_H_examples}.
\end{proof}

We consider the action of $\Phi_H$
on the category $\CC$ of $1$-dimensional
sheaves supported on fibers of $\pi$.
\begin{lemma}\label{phi_H_CC}
Let $A \in \CC$ be a semistable sheaf and let $\mu_0 = \mu(A)$ be its slope.
With the conventions $1/0 = \infty$ and $\infty/\infty = 1$
we have:
\begin{itemize}
\item If $\mu_0 \in [-1, \infty]$,
then $\widetilde{A} = G[-2]$ for a semistable sheaf $G \in \CC$ of
slope $-\mu_0/(1+\mu_0) \geq -1$.
\item If $\mu_0 \in (-\infty, -1)$, then
$\widetilde{A} = G[-1]$ for a semistable sheaf
$G \in \CC$ of slope~$-\mu_0/(1+\mu_0) < -1$.
\end{itemize}
In particular, for an interval $[a,b] \subset [-1, \infty]$ we have
\[ \Phi_H( \CC_{[a,b]} ) = \CC_{[-b/(1+b), -a/(1+a)]}[-2] \]
and likewise for open or half-open intervals in $[-1, \infty]$.
\end{lemma}
\begin{proof}
This follows from Proposition \ref{Prop_Action_on_Semistables_in_CC}
and the following:
If $A \in \CC$ is semistable of slope $a <\infty$, then
$\BD(A) = T[-2]$ for a semistable sheaf $T \in \CC$ of slope $-a$.
\end{proof}

\begin{lemma}\label{sublemma3}
Let $\CI_C$ be the ideal sheaf of a Cohen--Macaulay curve $C \subset X$
with no components supported on fibers of $\pi$.
Then $\widetilde{\CI_C}$ is a sheaf.
\end{lemma}
\begin{proof} For all $x \in X$ and $i \geq 1$ we claim the vanishing of
\begin{equation*}
\Hom( \widetilde{\CI_C}, \BC_x[-i] ) = \Hom( \imath_{s \ast} \Fm_x, \CI_C[2-i] )
\end{equation*}
with $s=\pi(x)$ and $\imath_s: X_s \hookrightarrow X$.
The case $i\geq 3$ is immediate. The case
$i=2$ follows since $\CI_C$ is torsion free. For $i=1$ we have
\[ \Ext^1( \imath_{s \ast} \Fm_x, \CI_C )
= \Hom( \imath_{s \ast} \Fm_x, \CO_C )
= \Hom_C( j^{\ast} \imath_{s \ast} \Fm_x, \CO_C ) \,,
\]
where $j : C \to X$ is the inclusion. Since $C$ has no
components supported on fibers, $j^{\ast} \imath_{s \ast} \Fm_x$
is zero-dimensional. Since $C$ is Cohen--Macaulay,
$\CO_C$ is pure and we conclude
$\Hom_C( j^{\ast} \imath_{s \ast} \Fm_x, \CO_C ) = 0$.
We conclude $\widetilde{\CI_C}$ is concentreated in degrees $\leq 0$.

We prove the lower bound.
By applying $\Phi_H$ to the exact sequence
\[ 0 \to \CI_C \to \CO_X \to \CO_C \to 0 \]
and taking the long exact sequence in cohomology
it suffices to prove $\widetilde{\CO_C}$
is concentrated in degrees $\geq 1$.

For all $\CL \in \Pic(S)$, $n > 1$ and $i \leq 0$
we have by Lemma \ref{Lemma_PHI_H_examples}
\begin{equation*}
\Hom( \pi^{\ast}\CL^{\vee}(-n S_0), \widetilde{\CO_C}[i] )
= \Hom( \CO_C, \pi^{\ast} \CL \otimes \CV_n [i+1] )
= 0,
\end{equation*}
since $\CO_C$ is supported on a curve and $\CV_{n}$ is locally free.
Since $\pi^{\ast} \CL (n S_0)$ can be
taken arbitrarily positive here,
$\widetilde{\CO_C}$ is concentrated in degrees $\geq 1$.
\end{proof}

\subsection{Two term complexes: the reduced case.}
\label{Section_Two_term_complexes}
Consider the following conditions on complexes
$I^{\bullet} \in D^b \Coh(X)$:
\begin{enumerate}
\item[(a)] $\ch( I^{\bullet} ) = (1, 0, -(H+dF), -n)$
for some $d \geq 0, n \in \BZ$.
\item[(b)] $I^{\bullet}$ is two-term:
$h^i( I^{\bullet} ) = 0$ whenever $i\neq 0,1$.
\item[(c)] $h^0( I^{\bullet} )$ is torsion free.
\item[(d)] $h^1( I^{\bullet} ) \in \CC_{[-1, \infty]}$.
\end{enumerate}

The aim of this subsection is to prove the following proposition.

\begin{prop}\label{MainProp}
Let $H \in \Pic(S)$ be a reduced effective class.
Then a complex $I^{\bullet} \in D^b \Coh(X)$ satisfies
Properties (a-d) above if and only if $\widetilde{I^{\bullet}}$ does.
\end{prop}

First we prove a lemma.
\begin{lemma} \label{sublemma1}
Let $D$ be an effective divisor on $S$, consider the surface
\[ W= \pi^{-1}(D) \]
and let $\CI_{C/W}$ be the ideal sheaf of a curve $C$ in $W$.
Assume the projection $C \to D$
is an isomorphism over a dense open
subset $U \subset D$, and the class of $C$
is $\iota_\ast D +mF \in H_2(X, \BZ)$ for some $m \in \BZ$.
Then \[ \widetilde{\CI_{C/W}} = \CE[-1] \]
for a $1$-dimensional sheaf $\CE$.
\end{lemma}
\begin{proof}
The proof proceeds in three steps.
\vspace{5pt}

\noindent \textbf{Step 1.}
Since $\CI_{C/W}(S_0)$ is a sheaf and 
$q : X \times_S X \to X$
has relative dimension $1$,
the complex
\[ A = \phi_{\CP}(\CI_{C/W}(S_0)) \]
is concentrated in degrees $0$ and $1$.
We show both $h^0(A)$ and $h^1(A)$ are of dimension $\leq 1$.
Since $C \to D$ is an isomorphism away from a $0$-dimensional
subset of $D$ we have
\[ \ch( \CI_{C/W} ) = (0, \pi^{\ast} D, -\iota_{\ast} D + aF, b) \]
for some $a,b \in \BZ$. By a direct calculation
this yields
\[ \ch( A ) = (0, 0, -\iota_{\ast} D + a'F, b') \]
for some $a',b'$. Hence it suffices to show $h^1(A)$ is supported on curves.

Let $U \subset D$ be the open subset over which $C \to D$
is an isomorphism.
In particular $\CO_C|_{\pi^{-1}(U)}$ is flat over $\CO_D|_{U}$.
For every $y \in X$ with $s = \pi(y) \in U$ we have by base change
\[ h^1(A) \otimes \BC_y
= H^1( X_s, \CI_{C/W}(S_0)|_{X_{s}} \otimes \CP_y )
\]
where $\CP_y = \phi_{\CP}(\BC_y)$ is the
sheaf corresponding to the point $y$.
Applying flatness we obtain
\[ h^1(A) \otimes \BC_y = H^1( X_s, \CI_{C \cap X_s/X_s}(S_0) \otimes \CP_y) \]
which is non-zero only if $y = C \cap X_s$. We conclude
\[ \mathrm{Supp}(h^1(A))\cap \pi^{-1}(U)=C\cap \pi^{-1}(U)\,, \]
and therefore that $h^1(A)$ is $1$-dimensional.

\vspace{6pt}
\noindent \textbf{Step 2.} We prove $\widetilde{I_{C/W}}$ is concentrated in degrees $\geq 1$.

Let $\CL = \CO_X(S_0 + \pi^\ast (H-K_S))$ and
let $G_i : = h^i(A)\otimes \CL$ for $i=0,1$.
Applying $\BD \circ \BT_{\CL}$ to the canonical exact triangle
\[
h^0(A) \rightarrow A \rightarrow h^1(A)[-1]
\]
yields the exact triangle
\begin{equation*}\label{eq3}
\mathbb{D}(G_1)[1] \rightarrow \widetilde{\CI_{C/W}} \rightarrow  \mathbb{D}(G_0).
\end{equation*}
By Step 1 the sheaves $G_i$ are of dimension $\leq 1$.
Hence the $\mathbb{D}(G_i)$ are concentrated in degrees $\geq 2$,
and so $\widetilde{\CI_{C/W}}$ is concentrated in degrees $\geq 1$.

\vspace{6pt}
\noindent \textbf{Step 3.} We show $\widetilde{\CI_{C/W}}$
is concentrated
in degrees $\leq 1$. 

By Lemma \ref{Lemma_PHI_H_examples}
we have for all $x \in X$
\begin{equation*}\label{eqq}
\Hom( \widetilde{\CI_{C/W}} , \BC_x[-i] )
=
\Hom( \imath_{s \ast} \Fm_x, \CI_{C/W}[2-i] ) \,,
\end{equation*}
which we claim vanishes for all $i \geq 2$.
The case $i\geq 3$ is immediate.
For the vanishing in case $i=2$,
we apply $\Hom( \imath_{s \ast} \Fm_x, \cdot )$ to the exact sequence
\[ 0 \to \CO_X(- \pi^{\ast} D) \to \CI_C \to \CI_{C/W} \to 0 \]
and use that $\imath_{s \ast} \Fm_x$ is supported in dimension $1$
and $\CI_C$ is torsion free.
\end{proof}

\subsection{Proof of Proposition \ref{MainProp}}
Assume a complex $I^{\bullet}$ satisfies (a-d).
Since $\Phi_H$ is an involution it is enough
to prove $\widetilde{I^{\bullet}}$ satisfies (a-d).

By Lemma \ref{Lemma_Chern_character} the complex
$\widetilde{I^{\bullet}}$ satisfies (a),
so we need to show (b,c,d).
We will make several reduction steps.

\vspace{6pt}
\noindent \textbf{Step 1.}
Since $h^1(I^{\bullet}) \in \CC_{[-1, \infty]}$
we have $\widetilde{h^1(I^{\bullet})} = G[-2]$
for some $G \in \CC_{[-1, \infty]}$.
Applying $\Phi_H$ to the canonical exact triangle
\[
h^0( I^{\bullet} ) \to I^{\bullet}
\to h^1( I^{\bullet} )[-1]
\]
and taking the long exact sequence in cohomology we therefore find
\[
0 \to h^0( \widetilde{I^{\bullet}} )
\to h^0\big( \widetilde{ h^0(I^{\bullet}) } \big)
\to G \to h^1( \widetilde{I^{\bullet}} )
\to h^1( \widetilde{ h^0(I^{\bullet}) } ) \to 0
\]
and $h^i( \widetilde{I^{\bullet}} ) = h^i( \widetilde{ h^0(I^{\bullet})})$
for $i \neq 0,1$.
Hence $\widetilde{I^{\bullet}}$ satisfies (b-d)
if $\widetilde{ h^0( I^{\bullet} ) }$ does.
We may therefore assume that
$I^{\bullet}$ is the ideal sheaf $\CI_C$
for some curve $C \subset X$.

\vspace{6pt}
\noindent \textbf{Step 2.}
Let $C' \subset C$ be the unique maximal Cohen--Macaulay subcurve.
We have the sequence
\[ 0 \to \CI_{C} \to \CI_{C'} \to Q \to 0 \]
where $Q$ is $0$-dimensional. Since $\widetilde{Q} = G[-2]$
for some $G \in \CC_{-1}$ we have
the exact sequence
\[ 0 \to h^1( \widetilde{\CI_{C'}} )
\to h^1( \widetilde{\CI_{C}} ) \to G \to 0 \,. \]
and $h^i( \widetilde{\CI_{C'}} ) = h^i( \widetilde{\CI_C} )$ for all $i \neq 1$.
Hence $\widetilde{\CI_C}$ satisfies
(b-d) if $\widetilde{\CI_{C'}}$ does,
and so we may assume $C$ is Cohen--Macaulay.

\vspace{6pt}
\noindent \textbf{Step 3.}
Let $D \subset S$ be the \emph{divisor} on $S$
defined by the image $\pi(C) \subset S$.
(In particular, $D$ remembers only
the codimension $1$ locus
of $\pi(C)$ and not isolated points or
thickening at points.)
We consider the curve
\[ C' = C \cap \pi^{-1}(D) \subset \pi^{-1}(D) \,. \]
Since $C' \subset C$, we have the exact sequence
\begin{equation} 0 \to K \to \CO_C \to \CO_{C'} \to 0 \,. \label{exact1} \end{equation}
with $K \in \CC$.

\vspace{4pt}
\noindent \emph{Claim:} $K$ is in $\CC_{[0, \infty)}$.

\vspace{3pt}
\noindent \emph{Proof of Claim.}
By construction $K$ is supported on fibers of $\pi$.
Moreover since $C$ is Cohen--Macaulay
(which is equivalent to $\CO_C$ is pure)
$K$ is pure of dimension $1$.

For our convenience let us assume $K$ is supported
over a point $s \in S$.

If $s \notin D$ then
$K$ is the structure sheaf of the connected component of $C$
that lies over $s \in S$.
It follows that $K$ admits a surjection $\CO_X \to K$.
Hence, if $K'$ is the the Harder-Narasimhan
factor of $K$ with smallest slope $\mu_0$, then
the natural composition $\CO_X \to K \to K'$ is
non-zero. Therefore $\mu_0 \geq 0$ by semistability,
and so $K \in \CC_{[0, \infty]}$, see
also the proof of \cite[Prop 6.8]{T12}.

If $s \in D$ then $K$ is supported over $F_n = \Spec \CO_X/ \pi^{\ast} \Fm_s^{n}$
for some $n \gg 0$ where $\Fm_s$ is the ideal sheaf of $s \in S$.
Let $f \in \CO_S$ be the local equation of $D$ near $s$, and let
$\overline{f}$ be its image in $\CO_S / \Fm_s^n$.
Consider the exact sequence
\begin{equation} 0 \to ( \overline{f} ) \to \CO_S / \Fm_s^n
\to \CO_S / (\Fm_s^n, f) \to 0 \,. \label{m1} \end{equation}
By flatness of $\pi$, \eqref{m1} remains exact under
pullback:
\begin{equation} 0 \to \pi^{\ast} ( \overline{f} ) \to \CO_X / \pi^{\ast} \Fm_s^n
\to \CO_X / \pi^{\ast} (\Fm_s^n, f) \to 0 \,. \label{m2} \end{equation}
Tensoring \eqref{m2} with $\CO_C$ we obtain
\[
\pi^{\ast} ( \overline{f} ) \otimes_{\CO_X} \CO_C
\twoheadrightarrow K
\hookrightarrow \CO_{C \cap F_n}
\to \CO_{C' \cap F_n} \to 0 \,.
\]
For every zero-dimensional sheaf $T$ on $S$
a surjection $\CO_S^k \to T \to 0$ for some $k \geq 1$
induces a surjection
$\CO_X^k \to \CO_C^k \to \pi^{\ast} T \otimes_{\CO_X} \CO_C$
(by pullback via $\pi^{\ast}$ and
tensoring with $\CO_C$).
Applying this to $T = (\overline{f})$
we obtain a composition of surjections
\begin{equation} \label{m3} \CO_X^k
\twoheadrightarrow \pi^{\ast} ( \overline{f} ) \otimes_{\CO_X} \CO_C
\twoheadrightarrow K \,.
\end{equation}
Let again $K'$ be the Harder-Narasimhan factor of $K$
with smallest slope.
Then composing \eqref{m3} with $K \to K'$
yields a non-zero section of $K'$. We find
$K' \in \CC_{[0, \infty)}$
and hence $K \in \CC_{[0, \infty)}$. \qed

We return to Step 3 of the reduction.
From \eqref{exact1} we obtain
\[ 0 \to \CI_C \to \CI_{C'} \to K \to 0 \,. \]
Since
$K \in \CC_{[0, \infty]}$
we have $\widetilde{K} = A[-2]$ for some
$A \in \CC_{(-1, \infty)}$.
Hence argueing as in Step 2 we find
$\widetilde{\CI_C}$ satisfies (b-d) if $\widetilde{\CI_{C'}}$ does.
Hence we may assume
$C$ is contained inside the surface $\pi^{-1}(D)$.

\vspace{6pt}
\noindent \textbf{Step 4.} We prove Properties (b) and (d).

Let $C'$ be the union of all
irreducible component of $C$ which are not supported
on fibers of $\pi$. The canonical exact sequence 
\[
0 \to \CI_C \to \CI_{C'} \to K \to 0
\]
with $K \in \CC$ yields the exact triangle
\begin{equation} \label{3353321}
\widetilde{K} \to \widetilde{\CI_{C'}} \to \widetilde{\CI_C} \,.
\end{equation}
By Lemma \ref{sublemma3} the complex
$\widetilde{\CI_{C'}}$ is a sheaf. Hence
taking long exact sequence of \eqref{3353321}
yields
$h^i( \widetilde{\CI_C} ) = 0$ for $i \neq 0,1$ (Property (b)),
and the isomorphism
\[
h^1(\widetilde{\CI_C}) \xrightarrow{\ \simeq \ } h^2(\widetilde{K}).
\]
Let $0 \to K_{{+}} \to K \to K_{{-}} \to 0$ be the unique
filtration of $K$ with $K_{+} \in \CC_{[-1, \infty]}$
and $K_{-} \in \CC_{[-\infty, -1)}$.
Then applying $\Phi_H$, using the long exact sequence in cohomology
and Lemma~\ref{phi_H_CC} yields
\[
h^2(\widetilde{K}) = h^2( \widetilde{K_{+}} )
\in \CC_{[-1, \infty]} \,. \]

\vspace{7pt}
\noindent \textbf{Step 5.} We prove that $h^0( \widetilde{I_C} )$
is an ideal sheaf, which implies Property (c).

Let $D$ be the divisor on $S$ as in Step 3
and recall that we may assume
\begin{equation} \label{eq_inclusion} C \subset W = \pi^{-1}(D) \,. \end{equation}
Since the class $H \in \Pic(S)$ is \emph{reduced}
we have $D \in |H|$. Hence the inclusion \eqref{eq_inclusion}
gives the exact sequence
\begin{equation*} 0 \to \CO_X(- \pi^{\ast} H) \to \CI_C \to \CI_{C/W} \to 0
\label{43413} \end{equation*}
where $\CI_{C/W}$ is the ideal sheaf of $C$ in $W$.
Applying $\Phi_H$ 
and taking cohomology yields
\begin{equation} \label{11111}
0 \to h^0(\widetilde{\CI_{C/W}}) \to h^0(\widetilde{\CI_C}) \to \CO_X \to
h^1( \widetilde{\CI_{C/W}}) \to h^1( \widetilde{\CI_C} ) \to 0 \,.
\end{equation}

Since $C$ has class $H + dF$ and $H$ is \emph{reduced},
the projection $C \rightarrow D$ is generically of degree $1$
on every component of $C$.
By Lemma \ref{sublemma1} the complex $\widetilde{\CI_{C/W}}[1]$
is a sheaf. Hence $h^0( \widetilde{\CI_C} )$ is an ideal sheaf.
\qed
\vspace{10pt}

The proof of Proposition~\ref{MainProp}
requires $H$ to be reduced only in Step 5.
Hence we obtain the following corollary
in case $H$ is not necessarily reduced.

\begin{cor} \label{Cor_MainProp} Let $H \in \Pic(S)$ be effective.
If $I^{\bullet} \in D^b \Coh(X)$
satisfies properties (a-d) above,
then $\widetilde{I^{\bullet}}$ satisfies (a), (b) and (d).
\end{cor}

\subsection{Purity}
Let $I^{\bullet} \in D^b \Coh(X)$ be a complex
which satisfies conditions
(a-d) of Section~\ref{Section_Two_term_complexes}.

\begin{lemma} \label{Lemma_Purity} Let $a \in [-1, \infty]$.
Then the following are equivalent:
\begin{itemize}
\item $\Hom( Q[-1], I^{\bullet}) = 0$ for all $Q \in \CC_{(a, \infty]}$, and
\item
$h^1( \widetilde{ I^{\bullet}} ) \in \CC_{[-a/(1+a), \infty ]}$.
\end{itemize}
The same holds when the intervals are replaced by
$[a, \infty]$ and $( -1/(1+a), \infty]$ respectively.
\end{lemma}
\begin{proof} 
By Lemma~\ref{phi_H_CC} we have
\[ \Hom( Q[-1], I^{\bullet}) = 0 \]
for all $Q \in \CC_{(a, \infty]}$ if and only if
\[ \Hom( \widetilde{I^{\bullet}} , G[-1] )
= \Hom( h^1( \widetilde{I^{\bullet}} ), G ) = 0 \]
for all $G \in \CC_{[-1, -a/(1+a) )}$. Since by
Corollary \ref{Cor_MainProp}
we have
$h^1( \widetilde{I^{\bullet}} ) \in \CC_{[-1,\infty]}$,
this is equivalent to
\[ h^1( \widetilde{I^{\bullet}} ) \in \CC_{[-a/(1+a), \infty]} \,. \]

The last claim of the Lemma follows by a parallel argument.
\end{proof}

We obtain the following corollary
which implies Theorem \ref{FM0}.
\begin{cor} Let $H \in \Pic(S)$ be reduced. Then a complex
$I^{\bullet}$ is a stable pair if and only if
$\widetilde{I^{\bullet}}$ is a $\pi$-stable pair.
\end{cor}
\begin{proof}
Let $I^\bullet \in D^b\mathrm{Coh}(X)$.
By Proposition \ref{MainProp} and by
Lemma~\ref{Lemma_Purity} with ${a=-1}$ and $a=\infty$
respectively we obtain:
\begin{enumerate}
\item[(1)] The complex $I^\bullet$ satisfies (a-d) (in Section 3.3) if and only if $\widetilde{I^\bullet}$ does.
\item[(2)] Suppose $I^\bullet$ satisfies (a-d), then $h^1(I^\bullet)$ is $0$-dimensional if and only if \[\mathrm{Hom}(Q[-1],\widetilde{I^\bullet})=0\] for all $\pi$-torsion sheaves $Q$.
\item[(3)] Suppose $I^\bullet$ satisfies (a-d), then $h^1(I^\bullet)$ is $\pi$-torsion if and only if \[\mathrm{Hom}(Q[-1],\widetilde{I^\bullet})=0\] for all $0$-dimensional sheaves $Q$.
\end{enumerate}
Hence $I^\bullet$ is a stable pair if and only $\widetilde{I^\bullet}$
is a $\pi$-stable pair. 
\end{proof}

\subsection{Non-reduced case}
We show by example that if the
class $H$ is non-reduced on $S$, then
there exist a stable pair $I^{\bullet}$ such that
$h^0(\widetilde{I^\bullet)}$ is not an ideal sheaf.
In particular the proof of Theorem~\ref{Theorem1} breaks down
and new methods are required, see \cite{OS2}.

\vspace{7pt}
\noindent \textbf{Example.}
Let $S$ be an elliptic K3 surface with a section $B \subset S$,
and let $E$ be an elliptic curve.
The Calabi--Yau 3-fold $X = S \times E$ admits
a trivial elliptic fibration over $S$
\[
\pi: X \rightarrow S
\]
by projection.
Let $F \subset S$ be a fixed fiber on the elliptic K3 surface $S$.
Assume $x_1$, $x_2$ are two distinct points on $E$,
and let $F_i=F \times \{x_i\}$, $i=1,2$.
We consider the curve
\[
C= \iota(B) \cup F_1 \cup F_2 \, \subset \, X
\]
in the curve class $\iota_\ast H$ where $H = B+2F$.
The class $H$ is a primitive, but non-reduced.

Let $\CI_C$ be the ideal sheaf of $C$ considered
as the stable pair
$[\CO_X \to \CO_C]$.
By Lemmas \ref{sublemma3} and \ref{Lemma_Chern_character},
the complex $\widetilde{\CI_C}$ is a sheaf with 
\begin{equation}\label{c_1}
c_1(\widetilde{\CI_C}) = 0 \in H^2(X,\BQ).
\end{equation}

We show that $\widetilde{\CI_C}$ is not an ideal sheaf.

Let $D= B+F$ and let $W = \pi^{-1}(D)$.
Since $C$ lies on the surface $W$
we have the exact sequence
\begin{equation*}\label{SES}
0 \to \CO_X(-\pi^\ast D) \to \CI_C \to \CI_{C/W} \to 0.
\end{equation*}
Since $\Phi_H( \CO_X(-\pi^\ast D) ) = \CO_X(-\pi^\ast F)$
we obtain
\[
0 \to h^0(\widetilde{\CI_{C/W}}) \to \widetilde{\CI_C} \xrightarrow{g} \CO_X(-\pi^\ast F) \to h^1(\widetilde{\CI_{C/W}}) \to 0.
\]
By definition, the sheaf $h^0(\widetilde{I_{C/W}})$
is supported on the surface $W$ and hence is torsion in $X$.
If $\widetilde{\CI_C}$ is an ideal sheaf, then $h^0(\widetilde{I_{C/W}})=0$
and thus $g$ is injective. This is a contradiction to (\ref{c_1})
since ideal sheaves are stable.

\section{Applications} \label{Section_Applications}
\subsection{Overview}
Here we prove Theorems \ref{Theorem_genus0},
\ref{Theorem_highergenus} and \ref{Theorem_highergenus_EulerChar},
and Proposition~\ref{Theorem_GW}.

\subsection{Proof of Theorem \ref{Theorem_highergenus}}
Let
$\mathsf{N}_{n,\beta} \in \BQ$
be the generalized Donaldson--Thomas invariant
counting semistable sheaves of class $(0,0, \beta, n)$.
In \cite{Br1, T16} the following structure result was proven:
\begin{equation} \label{TE}
\sum_{n,\beta}
\mathsf{P}_{n,\beta} q^n t^{\beta}  
=
\exp\Big(\sum_{n > 0, \beta>0} (-1)^{n-1} n \mathsf{N}_{n,\beta} q^n t^{\beta} \Big)
\cdot \sum_{n, \beta} \mathsf{L}_{n,\beta} q^n t^{\beta}
\end{equation}
where for every $\beta$
the invariants $\mathsf{L}_{n,\beta} \in \BZ$ satisfy
\begin{itemize}
\item $\mathsf{L}_{n,\beta} = \mathsf{L}_{-n,\beta}$ for all $n$,
\item $\mathsf{L}_{n,\beta} = 0$ for all $n \gg 0$.
\end{itemize}
For every $H$ we define the series
\[
f_H = \sum_{d \geq 0} \sum_{n > 0}
(-1)^{n-1} n \mathsf{N}_{n,H+dF} q^n t^d
\quad \text{and} \quad
L_H = \sum_{d \geq 0} \sum_{n \in \BZ } L_{n,H+dF} q^n t^d \,.
\]

Let $H \in \Pic(S)$ be a reduced class.
Picking out all $t^{H}$-terms in \eqref{TE} yields
\begin{equation} \label{452452}
\frac{
\PT_H(q,t)}{\PT_0(q,t)}
=
\sum_{k \geq 1} \sum_{\substack{H = H_1 + \ldots + H_k \\ H_i \text{effective}}}
\left( f_{H_1} \cdots f_{H_k} + \sum_{i=1}^{k}
\Big( \prod_{j \neq i} f_{H_j} \Big) \frac{L_{H_i}}{L_0}  \right) \,.
\end{equation}
Since every summand $H_i$ in the above sum is
non-zero and reduced
we have by \cite[2.9]{Tpar2}
\begin{equation} \label{eq502} f_{H_i} =
(q^{1/2} + q^{-1/2})^{-2}
\sum_{d \geq 0} \mathsf{N}_{1,H_i+dF} t^d
\,. \end{equation}
Moreover by \cite{T12} we have
$L_0 = \prod_{m \geq 1} (1-t^m)^{-e(S)}$,

Plugging both into \eqref{452452}
we find that for every $d$ the $t^d$ coefficient of $\PT_H/\PT_0$
can be written in the following form:
\[
\left[ \ \frac{\PT_H(q,t)}{\PT_0(q,t)} \ \right]_{t^d}
= \sum_{k= -n}^{r} a_{d,k} (q^{1/2} + q^{-1/2})^{2k}
\]
for some $r \gg 0$ (dependent on $d$) and where $n$ is
the largest integer for which there
exist a decomposition $H = \sum_{i=1}^{n} H_i$
into effective classes.

On the Jacobi form side we have for every $i$
\[ \phi_{-2,1}^{i-1} \phi_{0,1}^{h-i} =
12^{h-i} (q^{1/2} + q^{-1/2})^{2i-2}
+ \sum_{j =i}^{N} a_j (q^{1/2} + q^{-1/2})^{2j} + O(t) \]
for some $a_j \in \BQ$ and some $N > 0$.
Also every $t^d$-coefficient
is of the form
\[
\left[ \ \phi_{-2,1}^{i-1} \phi_{0,1}^{h-i} \ \right]_{t^d}
= \sum_{k= i-1}^{r'} a'_{d,k} (q^{1/2} + q^{-1/2})^{2k} \,.
\]

Therefore by an induction argument
there exist power series
\[ f_{-(n-1)}(t), \ldots, f_{h}(t) \in \BQ[[t]] \]
such that every $t^d$-coefficient of the series
\[ \CR(q,t) =
\frac{\PT_H}{\PT_0}
- \sum_{i=-(n-1)}^{h} f_i(t) \varphi_{-2,1}^{i-1} \varphi_{0,1}^{h-i}
\]
takes the form
\begin{equation} \CR_d(q,t) = \left[ \ \CR(q,t) \ \right]_{t^d}
=
\sum_{k = h}^{r} b_{d,k} (q^{1/2} + q^{-1/2})^{2k} \,.
\label{aaaedfdfsd}
\end{equation}
Equivalently, every coefficient of $\CR_d(q,t)$
vanishes to order $u^{2h}$
when expanded in the variables $q=-e^{iu}$.
We prove $\CR(q,t) = 0$.

By \eqref{aaaedfdfsd} and since
$h \geq 0$ by assumption
we have the equality of power series
\begin{equation} \CR(q^{-1}, t) = \CR(q,t) \label{eq500} \end{equation}

Using the definition of $\phi_{-2,1}$ as a product one verifies
\[ \frac{1}{\phi_{-2,1}(q^{-1} t, t)}
= q^{-2} t^{1} \frac{1}{\phi_{-2,1}(q,t)}
\]
\emph{as power series}.
By the Jacobi form property of $\phi_{0,1}$ \cite{EZ} we also have
\[ \phi_{0,1}(q^{-1} t, t)
= q^{2} t^{-1} \phi_{0,1}(q,t) \,.
\]
Hence we obtain for every $i \leq h$ the equality of power series
\[ 
\big( \phi_{-2,1}^{i-1} \phi_{0,1}^{h-i} \big)(q^{-1}t, t)
=
q^{2(h-1)} t^{-(h-1)}
\big( \phi_{-2,1}^{i-1} \phi_{0,1}^{h-i} \big)(q,t) \,.
\]
Applying Theorem~\ref{Theorem1} we thus find
the equality of power series:
\begin{equation} \label{eq501}
\CR(q^{-1}t, t)
=
q^{2(h-1)} t^{-(h-1)}
\CR(q,t) \,.
\end{equation}

Combining \eqref{eq500} and \eqref{eq501}
we conclude
\begin{equation}
\CR(qt^{\lambda}, t) = t^{- (h-1) \lambda^2} q^{-2 (h-1) \lambda} \CR(q,t)
\label{Masfsaf}\end{equation}
for all $\lambda \in \BZ$ as power series.
Let $c_{n,d}$ be the coefficient of
$q^n t^d$ in $\CR(q,t)$.
Then \eqref{Masfsaf} is equivalent to
\begin{equation} c_{n,d}
= c_{n + 2(h-1) \lambda, d + \lambda n + (h-1) \lambda^2}
\label{Masfsd2} \end{equation}
for all $n,d, \lambda \in \BZ$.

Assume $\CR(q,t)$ is non-zero and let
$d$ be the smallest integer such that
$\CR_d(q,t)$ is non-zero.
Since the sum in \eqref{aaaedfdfsd}
starts at $k=h$ we have
\[ c_{n,d} \neq 0 \]
for some $n \geq h \geq 0$.
But then by \eqref{Masfsd2} with $\lambda=-1$
we obtain
\[ c_{n,d} = c_{n-2h+2, d-n+(h-1)} \neq 0 \,. \]
Since $d-n+(h-1) < d$ this contradicts the choice of $d$. \qed

\subsection{Proof of Theorem \ref{Theorem_genus0}}
By Theorem~\ref{Theorem_highergenus} we have
\[ \frac{\PT_H(q,t)}{\PT_0(q,t)} = f_0(t) \frac{1}{\phi_{-2,1}(q,t)} \]
for some power series $f_0(t)$. Hence
we need to prove $f_0(t) = \mathsf{F}_H(t)$.

By \eqref{452452} and \eqref{eq502},
and since the genus $0$ Gopakumar--Vafa invariant
is
$\mathsf{n}_{\beta} = \mathsf{N}_{1, \beta}$
we have
\[ \frac{\PT_H(q,t)}{\PT_0(q,t)} = \mathsf{F}_H(t) \frac{q}{(1+q)^2}
+ \frac{L_H}{L_0} \,. \]
Similarly,
\[ \frac{1}{\phi_{-2,1}(q,t)}
= \frac{q}{(1+q)^2} + \phi'(q,t) \]
where $\phi'(q,t)$ is regular at $q=-1$. Comparing
both sides we obtain
\[
\pushQED{\qed} 
f_0(t) = \mathsf{F}_H(t) \,. \qedhere
\popQED
\]

\subsection{Proof of Proposition~\ref{Theorem_GW}}
For every $i$ we have under the variable change $q=-e^{iu}$
\[
\phi_{-2,1}^{i-1} \phi_{0,1}^{h-i} = 12^{h-i} u^{2i-2} + O(u^{2i}) \,. 
\]
Hence the $u^{2i}$-coefficients of $\PT_H/\PT_0$ for $i < h$
uniquely determine the functions $f_k(t)$
in Theorem~\ref{Theorem_highergenus}.

If $H$ is irreducible the coefficient of $u^H$ in
\[
\log \Big( 
\sum_{L, n,d} \mathsf{P}_{n,L+dF} q^n t^{d}u^{L} 
\Big) \]
is exactly $\PT_H/\PT_0$.
Therefore Proposition~\ref{Theorem_GW}
follows from
the GW/PT correspondence:
\[
\pushQED{\qed} 
\frac{\PT_H(q,t)}{\PT_0(q,t)}
= \sum_{g=0}^{\infty} 
\mathrm{GW}_H^{g}(t) u^{2g-2} \,.
\qedhere
\popQED
\]

\subsection{Proof of Theorem \ref{Theorem_highergenus_EulerChar}}
By the same argument as for Theorem \ref{Theorem_highergenus},
but using the Euler characteristics case \cite{T08}
of the structure result \eqref{TE}. \qed

\section{Examples} \label{Section_Examples}
\subsection{Overview}
We give examples
of our results for the
Schoen Calabi--Yau (Section~\ref{Section_Schoen_Calabi-Yau}),
the STU-model (Section~\ref{STU-Model}),
the local case (Section~\ref{Section_local_case}),
and the product case $S \times E$ (Section~\ref{Section_product_case}).
For $\mathrm{K3} \times E$
we give the proof of Theorems
\ref{Theorem_K3xE} and \ref{Theorem_Bryan}.
For $\mathrm{A} \times E$
we reprove several of the results of \cite{BOPY}
and connect a conjecture of \cite{BOPY}
to the modularity of the series $f_k(t)$.
We end in Section~\ref{Section_Examples_Naive_Euler_Char}
with an example that the Noether--Lefschetz terms
$\widetilde{f}_k(t)$ for unweighted Euler characteristics
may not be modular.

\subsection{The Schoen Calabi--Yau}
\label{Section_Schoen_Calabi-Yau}
A rational elliptic surface $R$ is the
blowup of $\p^2$ along the $9$ intersection
points of two distinct integral cubics. The associated
pencil of cubics
defines an elliptic fibration
\[ R \to \p^1 \,. \]
Let $R_1, R_2$ be two generic rational elliptic surfaces. The Schoen
Calabi--Yau is the fiber product
\[ X = R_1 \times_{\p^1} R_2 \,. \]
We consider $X$ to be elliptically fibered via the projection
\[ \pi : X \to R_1 \]
to the first factor.

Let $B$ be the class of a section of $R_1 \to \p^1$.
By a calculation of Bryan and Leung \cite[Thm 6.2]{BL}
the genus $0$ Gromov--Witten invariants of $X$
in classes~$B+dF$ are
\[
\GW^0_{B}(t)
=
\sum_{d \geq 0} \GW^{R_2}_{0, B+dF} t^d
=
\prod_{m=1}^{\infty} (1-t^m)^{-12}
\,.
\]
Since $e(X) = 0$ and $e(R_1) = 12$ we also have
\[ \PT_0(q,t) = \prod_{m \geq 1} (1-t^m)^{-12} \,. \]
Using Theorem \ref{Theorem_genus0} we conclude
\[ \PT_B(q,t) =
\frac{1}{(q + 2 + q^{-1})} \prod_{m \geq 1}
\frac{1}{(1+q t^m)^2 (1-t^m)^{20} (1+q^{-1} t^m)^{2}} \,. \]

Let $E$ be the class of an elliptic fiber of $R_1 \to \p^1$.
We have 
\[ \GW^0_{E}(t) = \GW_{0,E} = 0, \quad
\GW^{1}_{E}(t) = \GW_{1,E} = 12 \,. \]
Hence in agreement with a stable pair calculation by J.~Bryan
we obtain:
\[ \PT_E(q,t) = 12 \prod_{m \geq 1} (1-t^m)^{-12} \,. \]

\subsection{The STU model}
\label{STU-Model}
The STU model is a particular non-singular
Calabi--Yau threefold $X$ which admits an elliptic fibration
\[ \pi : X \to \p^1 \times \p^1 \]
with a section.
An explicit description of $X$ as a hypersurface
in a toric variety can be found in \cite{KMPS}.
The composition of $\pi$
with the projection of $\p^1 \times \p^1$ to the $i$-th factor
\begin{equation} \pi_i : X \to \p^1 \times \p^1 \to \p^1 \label{K3fibration}
\end{equation}
is a family of K3 surfaces with $528$ singular fibers.
We identify
\[ \Pic( \p^1 \times \p^1 ) = \BZ \oplus \BZ \,. \]

Recall the genus $0$ potential $\mathsf{F}_H = \mathsf{F}_H(t)$.
By\footnote{We assume the GW/PT correspondence for $X$ here. We expect the proof to be similar to \cite{PaPix2}.} \cite{KKRS,KMPS} we have
\[ \mathsf{F}_{(1,0)}(t) = \mathsf{F}_{(0,1)}(t)
= \frac{-2 E_4(t) E_6(t)}{\prod_{m \geq 1}(1-t^m)^{24}} \,. \]
The term $-2 E_4(t) E_6(t)$
is the Noether--Lefschetz
contribution of the fibration \eqref{K3fibration}
and counts the jumping of the Picard rank
of the fibers of $\pi_i$.
The term
\[ \prod_{m \geq 1} \frac{1}{(1-t^m)^{24}} \]
is the generating series of (primitive) genus $0$
invariants of K3 surfaces, and
counts rational curves
in the fibers of $\pi_i$.
The series $\mathsf{F}_{(1,0)}(t)$
splits naturally as a product of both contributions.
Applying Theorem~\ref{Theorem_genus0} we obtain
\[ \frac{\PT_{(1,0)}(q,t)}{\PT_0(q,t)} =
\frac{ -2 E_4(t) E_6(t) }{
\Delta(t) \phi_{-2,1}(q,t) } \,, \]
in perfect agreement with the
higher genus GW/NL correspondence \cite{GWNL}
and the Katz--Klemm--Vafa formula \cite{KKV, MPT}.

The formula for the class $H = (1,1)$ is new and more interesting,
since curves in the base may degenerate
to the union of two lines. 
By Theorem~\ref{Theorem_highergenus}
we have
\[
\frac{\PT_{(1,1)}(q,t)}{\PT_0(q,t)}
\, = \,
f_{-1}(t)
\cdot \frac{ \wp(q,t) }{\phi_{-2,1}(q,t) } + 
f_0(t) \cdot \frac{1}{\phi_{-2,1}(q,t)} \,.
\]

Studying the polar terms in \eqref{452452} yields
\[ f_{-1}(t) = \mathsf{F}_{(1,0)} \mathsf{F}_{(0,1)} =
4 E_4(t)^2 E_6(t)^2 \prod_{m \geq 1}(1-t^m)^{-48} \]
and 
\[ f_0(t)
= \mathsf{F}_{(1,1)} + \frac{1}{12} E_2(t)
\mathsf{F}_{(1,0)} \mathsf{F}_{(0,1)} \,. \]
By the calculations of \cite[6.10.5]{KKRS} we have\footnote{
Since $X$ is a hypersurface in a
toric variety, mirror symmetry for $X$
in genus $0$ is proven.
Hence there is no
difficulty in making the genus $0$ calculations
of \cite{KKRS} rigorous.},
\[
\mathsf{F}_{(1,1)}(t)
=
-E_4 E_6 \left(
\frac{67}{36} E_4^3 + \frac{65}{36} E_6^2 + \frac{1}{3} E_2 E_4 E_6
\right)
\prod_{m \geq 1}(1-t^m)^{-48} \,.
\]
Hence we find
\[
f_0(t)
=
- E_4 E_6  \left(
\frac{67}{36} E_4^3 + \frac{65}{36} E_6^2 \right)
\prod_{m \geq 1}(1-t^m)^{-48} \,.
\]
In particular, both $f_0(t)$ and $f_{-1}(t)$
are modular forms\footnote{We ignore here the non-modularity
arising from the missing $t^{-2}$ factor in the Euler product.}.
We conclude that $\PT_{(1,1)}/\PT_{0}$ is a meromorphic Jacobi form.

The elliptic surface over a non-singular line $L \subset \p^1 \times \p^1$
of class $(1,1)$ has 48 nodal fibers.
Hence by \cite{BL} we may interpret the term
\[ \prod_{m \geq 1} \frac{1}{(1-t^m)^{48}} \]
as counting rational curves in the surfaces $\pi^{-1}(L)$.
The factor
\[ 
4 E_4(t)^2 E_6(t)^2
\]
in $f_{-1}(t)$ is the product of the Noether--Lefschetz
series of the family \eqref{K3fibration},
and arises here naturally from the Noether--Lefschetz data
of the $2$-dimensional family of surfaces $\pi^{-1}(L)$
over broken lines $L = L_1 \cup L_2$.
We expect the factor
\[ -E_4 E_6  \left(
\frac{67}{36} E_4^3 + \frac{65}{36} E_6^2 \right)
=
-\frac{11}{3} + 1448q - 362376q^{2} + 85977632q^{3}
+ \ldots
\]
to arise from the Noether--Lefschetz theory
of the full family of elliptic surfaces
over curves in the base.

\subsection{The local case}
\label{Section_local_case}
Let $X$ be a non-singular
projective threefold with $\omega_X \simeq \CO_X$.
In particular, we allow $H^1(X, \CO_X)$ to be non-zero.
Assume that $X$ admits an elliptic fibration
\[ \pi : X \to S \]
with integral fibers over a non-singular surface $S$, and let
$\iota : S \to X$
be a section.
We also fix a Cohen--Macaulay curve
\begin{equation} C \subset X \label{CMcurve} \end{equation}
which satisfies the following conditions:
\begin{itemize}
\item The curve $C$ has no components supported on fibers of $\pi$.
\item The curve $C$ is reduced.
\item The restriction
\[ \pi|_{C} : C \to D := \pi(C) \]
is an isomorphism away from a zero-dimensional subset $T \subset D$.
\end{itemize}
In particular,
the image $D=\pi(C)$ is a reduced divisor in $S$.

Consider the closed subset
\[ P_{n}(X,C,d) \subset P_n(X,[C]+dF) \]
of stable pairs $\CO_X \to \CF$
such that the support of $\CF$ contains $C$.
We define $C$-local pairs invariants by
\[ \mathsf{P}_{n,C,d} = \int_{P_{n}(X,C,d)} \nu \dd{e} \,. \]
where $\nu$
is the restriction of the
Behrend function of $P_n(X,[C]+dF)$.
Let
\[ \PT_C(q,t)
= \sum_{d = 0}^{\infty} \sum_{n \in \BZ}
\mathsf{P}_{n,C,d} q^n t^d,
\]
be the generating series of $C$-local pairs invariants.

Recall also the generating series
$\PT_0(q,t)$ from \eqref{Toda_result}.

\begin{thm} \label{Theorem_LOCAL_CASE}
Let $h$ be the arithmetic genus of the
divisor $D \subset S$. Then in the situation above,
\[
\frac{ \PT_C(q^{-1} t , t) }{ \PT_0(q^{-1}t,t) }
=
q^{2(h-1)} t^{-(h-1)}
\frac{ \PT_C(q , t) }{ \PT_0(q,t)} \,.
\]
Assume $h \geq 0$,
and let $n$ be the number of rational
components of $C$.
Then there exist power series
$f_i(t) \in \BQ[[t]]$ such that
\[
\frac{\PT_C(q,t)}{\PT_0(q,t)}
\ = \  
\sum_{i=-(n-1)}^{h} f_i(t) \cdot \phi_{-2,1}(q,t)^{i-1} \phi_{0,1}(q,t)^{h-i} \,.
\]
\end{thm}
\begin{proof}
We apply the same argument
as in the proof of Theorem~\ref{Theorem1}.
The wall-crossing argument
follows exactly the discussion
of Section~\ref{Section_pi_stable_pairs}
but using the moduli space of
($\pi$-)stable pairs $I^{\bullet}$
such that the curve defined by
$h^0(I^{\bullet})$ contains $C$.
Also the argument using the derived equivalence $\Phi_H$
is parallel. We only need to check that
the support of a stable pair $I^{\bullet}$
contains $C$ if and only if
so does the curve defined by $h^0( \widetilde{I^{\bullet}})$.
This reduces to showing that $\widetilde{\CI_C} = \CI_C$.
But this follows from \eqref{11111},
the proof of Lemma~\ref{sublemma1}
and since $C$ is reduced.
Hence we conclude the first claim.

The second claim follows
exactly parallel to the proof
of Theorem~\ref{Theorem_highergenus},
with \eqref{TE}
replaced by the $C$-local case
\cite[Thm 4]{O1}
and using \cite[Lem 2.12]{Tpar2}.
\end{proof}

\subsection{The product case}
\label{Section_product_case}
Let $S$ be a non-singular projective surface with $\omega_S \simeq \CO_S$.
Hence $S$ is a K3 surface or an abelian surface.
Let $E$ be an elliptic curve and consider the Calabi--Yau threefold
\[ X = S \times E \]
elliptically fibered over $S$
by the projection
to the first factor.
Let $0_E \in E$ denote the zero, and fix the section
\[ \iota : S = S \times 0_E \hookrightarrow X \,. \]

Let $H \in \Pic(S)$ be a divisor in the base,
and let 
\begin{equation} P_n(X, (H, d)) \subset P_n(X, H+dF) \label{mdfdfe}
\end{equation}
be the moduli space of stable pairs $(\CF,s)$
of class $H+dF$ such that
the pushforward $\pi_{\ast}[C]$ of the
cycle of the support $C$ of $\CF$ lies in $|H|$.

The group $(E,0_E)$ acts on the moduli space
by translation with finite stabilizers.
We define reduced stable pair invariants of $X$ by the
Behrend function weighted Euler characteristic
\[ \mathsf{P}^{\text{red}}_{n, (H,d)}
=
\int_{P_n(X, (H,d)) / E} \nu \dd{e} \,,
\]
where the Euler characteristic is taken
in the orbifold sense.
For $S$ an abelian surface
the definition first appeared in
\cite{Gul}, for $S$ a K3 surface
it can be found in \cite{Bryan-K3xE}.
Define the
generating series
of reduced invariants
\[
\PT^{\mathrm{red}}_{H}(q,t)
= \sum_{d \geq 0} \sum_{n \in \BZ}  
\mathsf{P}^{\text{red}}_{n, (H,d)} q^n t^d \,.
\]

\begin{thm} \label{THM_SxE}
Let $H \in \Pic(S)$ be irreducible of arithmetic genus $h$.
There exist power series
$f_i(t) \in \BQ[[t]]$ such that
\[
\frac{ \PT^{\mathrm{red}}_{H}(q,t) }{\PT_0(q,t)}
\ = \  
\sum_{i=0}^{h} f_i(t) \phi_{-2,1}(q,t)^{i-1} \phi_{0,1}(q,t)^{h-i} \,.
\]
\end{thm}
\begin{proof} This follows directly
from Theorem~\ref{Theorem_LOCAL_CASE}
by integration over the quotient of
the Chow variety of curves by the translation action,
compare \cite[4.11]{O1}.
\end{proof}

\subsection{The case $\mathrm{K3} \times E$}
Let $S$ be a projective K3 surface and let
\[ H \in \Pic(S) \]
be a primitive class of arithmetic genus $h$.
By deformation invariance \cite{O1}, 
the series $\PT^{\mathrm{red}}_{H}$
only depends on $h$. Hence we can assume $H$ is irreducible.
We write
\[ \PT^{\mathrm{K3} \times E}_{h}(q,t)
= \PT^{\mathrm{red}}_{H}(q,t) \,. \]

\begin{proof}[Proof of Theorem~\ref{Theorem_K3xE}]
Since $e(S) = 24$ and $e(X) = 0$ we have by \eqref{Toda_result},
\[ \PT_0(q,t) = \frac{1}{\Delta(t)} \,. \]
Hence by Theorem~\ref{THM_SxE}
it suffices to show that $f_i(t)$
are quasi-modular forms
of weight $2i$.
By \cite[Thm.1]{O1} the
invariant $\mathsf{P}^{\text{red}}_{n, H+dF}$
is equal to the stable pair invariant of $S \times E$
defined via the reduced virtual class and insertions.
Hence by \cite[Prop. 5]{K3xE} we have the GW/PT correspondence
\begin{equation} \label{434134}
\frac{\PT^{\mathrm{K3} \times E}_{h}(q,t)}{ \PT_0(q,t) }
=
\sum_{d \geq 0}
\sum_{g \geq 0} \mathsf{\GW}_{g,H+dF} u^{2g-2} t^d
\end{equation}
under the variable change $q = -e^{iu}$;
here $\mathsf{\GW}_{g,H+dF}$
is the reduced (connected) genus~$g$
Gromov--Witten invariant
of $S \times E$
in class $H+dF$
defined via the insertion $\tau_0(H^{\vee} \boxtimes \mathrm{pt})$, see \cite{K3xE}.
Define
\[ \mathrm{GW}_H^{g}(t) = \sum_{d \geq 0}
\mathsf{\GW}_{g,H+dF} t^d \,. \]
By \eqref{434134} the series $f_i(t)$ are determined by the identity
\begin{equation*}
\sum_{i=0}^{h} f_i(t) \phi_{-2,1}(q,t)^{i-1} \phi_{0,1}(q,t)^{h-i}
\equiv
\sum_{g=0}^{h} 
\mathrm{GW}_H^{g}(t) u^{2g-2} \ \ \mathrm{mod}\ u^{2h}
\end{equation*}
under the variable change $q=-e^{iu}$.
Hence it is enough to show that
\[ \mathrm{GW}_H^{g}(t) \in \BQ[[t]] \]
are quasi-modular forms of weight $2g$.

By a degeneration argument \cite[5.3]{MPT}
the series $\mathrm{GW}_H^{g}(t)$ can be expressed
in terms of the Gromov--Witten invariants of $R \times E$,
where $R$ is an rational elliptic surface.
Since $R$ is deformation equivalent to a toric surface, the Gromov--Witten classes of $R$ are tautological on $\overline{M}_{g,1}$ \cite{FPrel}.
Hence by Behrend's product formula \cite{B2} the generating series $\mathrm{GW}_H^g$ can be expressed in terms of the following generating series of elliptic Gromov--Witten invariants:
\begin{equation}\label{taut}
\sum_{d\geq 0} \left(  \int_{[ \Mbar_{g,1}(E, d) ]^{\text{vir}}} \mathrm{ev}^\ast(\mathrm{pt})\cup h^\ast \alpha  \right) q^d,
\end{equation}
where $\mathrm{ev}: \Mbar_{g,1}(E, d) \rightarrow E$ is the evaluation map, $h: \Mbar_{g,1}(E, d) \rightarrow \Mbar_{g,1}$ is the forgetful map, and $\alpha$ is a tautological class on $\Mbar_{g,1}$.

The induction process in the proof of \cite[Prop 29]{MPT} for K3 surfaces can be identically applied to the elliptic curve $E$, which allows us to write the tautological integrals (\ref{taut}) in terms of descendent integrals. By \cite[Sec.5]{OP1}
we find (\ref{taut}) is a quasi-modular form of weight $2g$,
see also \cite[Prop 28]{MPT}.
\end{proof}

\begin{proof}[Proof of Theorem~\ref{Theorem_Bryan}]
Let $H_0 \in \Pic(S)$ be irreducible of genus $0$. Then
\[ \GW^0_{H_0}(t)
= \mathsf{\GW}_{0, H_0} = 1 \] 
so by Theorem~\ref{Theorem_genus0} and
$\PT_0 = \Delta^{-1}$ we obtain
\[
\PT_{H_0}(q,t) = 1 \cdot \frac{1}{\Delta(t)} \cdot
\frac{1}{\phi_{-2,1}(q,t)} \,.
\]

Let $H_1 \in \Pic(S)$ be an irreducible class of genus $1$.
Then
\begin{align*}
\GW^{0}_{H_1}(t)
& = \mathsf{\GW}_{0, H_1} = 24 \\
\GW_{1,{H_1}+dF} & = 24 \sigma(d) \cdot
\langle \lambda_1 \rangle^S_{1, H_1}
 = \sigma(d) \cdot (H_1^2) \cdot \langle 1 \rangle^S_{0, H_1}
 = 0
\end{align*}
where we used the Gromov--Witten bracket notation,
the product formula,
and $H_1 \cdot H_1 = 0$ in the second line.
Hence by Theorem~\ref{THM_SxE} and
\eqref{gen1_eqn}
we find

\[
\PT_{H_1}(q,t)
=
24 \cdot \frac{1}{\Delta(t)} \cdot \wp(q,t) \,.
\qedhere \]
\end{proof}

\subsection{Abelian threefolds}
Let $A$ be an abelian surface and let
\[ H \in \Pic(A) \]
be an irreducible class of arithmetic genus $h \geq 2$.
By the same argument as in \cite{O1}
the series $\PT^{\mathrm{red}}_{H}$
is invariant under deformations
which preserve the product structure $A \times E$
and keep $H$ of Hodge type $(1,1)$ on $A$.
Hence the series only depends on the arithmetic genus $h$
and we write
\[ \PT^{\mathrm{A} \times E}_{h}(q,t)
= \PT^{\mathrm{red}}_{H}(q,t) \,. \]

In \cite{BOPY} the Euler characteristic version
of the series $\PT^{\mathrm{A} \times E}_{h}$
was computed in cases $h \in \{2,3\}$ via cut-and-paste methods,
and a full formula was conjectured for any $h$.
Using calculations on the Gromov--Witten side
the case $h=2$ was proven in \cite{K3xP1}
also for the weighted case.
Here we reprove the result for $h=2$,
and for $h=3$ when assuming the GW/PT correspondence.

\begin{prop} \label{AxE_Prop1}
$\PT^{\mathrm{A} \times E}_{2}(q,t) = \phi_{-2,1}(q,t)$
\end{prop}

\begin{proof}[Proof of Proposition \ref{AxE_Prop1}]
By Theorem~\ref{THM_SxE} we have
\[ \PT^{\mathrm{A} \times E}_{2}(q,t) =
f_0(t) \frac{\phi_{0,1}^2}{\phi_{-2,1}}
+ f_1(t) \phi_{0,1} + f_2(t) \phi_{-2,1}\,. \]
By a degeneration argument\footnote{
The reduced virtual class splits naturally
under degenerating the elliptic factor to a nodal rational curve
and resolving,
compare \cite[2.5]{K3xP1} and \cite[Section 4]{MPT}.
} the GW/PT
correspondence holds for ${h=2}$. 
Hence we may compute $f_i(t)$ by Gromov--Witten theory.
Let $H_2 \in \mathrm{Pic}(A)$ be irreducible of genus 2, and let
\[ \GW_{H_2}^g(t) = \sum_{d \geq 0} \GW^{A \times E}_{g,H_2+dF}t^d \]
denote the generating series of reduced genus $g$ Gromov--Witten
invariants of $A\times E$ following \cite[7.2]{BOPY}.
Since a generic abelian surface does not admit nontrivial maps
from rational or elliptic curves we have
$\GW_{H_2}^{0}(t) = \GW_{H_2}^{1}(t) = 0$.
The genus $2$ invariants are
\[ \GW_{H_2}^2(t) = 1 \]
since there is exactly one genus $2$ curve in the linear system $|H_2|$ on $A$.
Applying \eqref{gen2_eqn} we conclude the result.
\end{proof}

Consider the theta function of the D4 lattice
\[ \vartheta_{\mathsf{D4}}(t)
=
1 + 24 \sum_{d \geq 1} \sum_{\substack{k|d \\ k \textup{ odd}}} k t^d
\,. \]
\begin{prop} \label{AxE_Prop2}
Let $H \in \Pic(A)$ be irreducible of arithmetic genus $3$,
and assume the GW/PT correspondence in the sense of \cite[Conj. B]{BOPY}
holds for the classes $H+dF$. Then
\[
\PT^{\mathrm{A} \times E}_{3}(q,t) =
12 \cdot \wp(q,t) \phi_{-2,1}(q,t)^2
- \vartheta_{\mathsf{D4}}(t) \cdot \phi_{-2,1}(q,t)^2 \,.
\]
\end{prop}

\begin{proof}[Proof of Proposition \ref{AxE_Prop2}]
By assumption we again only need to determine the
Gromov--Witten invariants $\GW_{H}^{g}(t)$ for genus up to $3$.
The invariants vanish in $g=0,1$.
By counting genus $2$ curves in $A$ we obtain
$\GW_{H}^{2}(t) = 12$.
For genus $3$, by using lattice counting
and the multiple cover formula \cite{BOPY} we obtain
\[ \GW_{H}^{3}(t) = \sum_{d\geq 1} \big( 8 \sigma(2d)
+ 64 \sigma(d/2) \delta_{d,\text{even}} \big) t^d \]
where $\sigma(d) = \sum_{k|d} k$.
The result follows as before by Proposition~\ref{Theorem_GW}.
\end{proof}

For general $h$ applying Theorem~\ref{THM_SxE}
yields the expansion
\[
\PT^{\mathrm{A} \times E}_{h}(q,t) = \sum_{i=2}^{h}
f_i(t) \phi_{-2,1}(q,t)^{i-1} \phi_{0,1}(q,t)^{h-i} 
\]
for power series $f_i(t)$.
By (and in agreement with) Conjecture C of \cite{BOPY}
we have the following.\\

\noindent \textbf{Conjecture \cite{BOPY}.}
Every series $f_{i}(t),\,  i=2, \ldots, h$ is
a modular form of weight $2i-4$ for
the congruence subgroup $\Gamma_0(h-1)$.
\vspace{5pt}

\subsection{Na\"ive Euler characteristic}
\label{Section_Examples_Naive_Euler_Char}
Consider as above $X = A \times E$
for an abelian surface with an irreducible
class $H \in \Pic(A)$ of arithmetic genus $3$.
Let
\[ \widetilde{\PT}_{H}(p,t)
= \sum_{d \geq 0} \sum_{n \in \BZ}   
e\big( P_n(X, (H, d)) / E \big) q^n t^d
\]
be the generating series of Euler characteristics
of (the quotient of) the stable pair
moduli space defined in \eqref{mdfdfe}.
By \cite[Thm. 5]{BOPY} we have
\[
\widetilde{\PT}_{H}(p,t)
=
12 \wp(-p,t) \phi_{-2,1}(-p,t)^2
+ \widetilde{f}_3(t) \phi_{-2,1}(-p,t)^2 \]
where
\[
\widetilde{f}_3(t) = 1 + 12 \sum_{d \geq 1} \sum_{k|d} ( 2k t^d + k t^{2d} )
= \frac{5}{2} - E_2(t) - \frac{1}{2} E_2(t^2) \,.
\]
In particular, $\widetilde{f}_3(t)$ is not modular.

\end{document}